\def\version{16/04/2017 -- version 8
\hfill\href{http://arxiv.org/abs/1503.05902}{arXiv:1503.05902}
}

\documentclass[11pt,a4paper,centertags]{amsart}
\usepackage{fullpage}
\usepackage[hypertexnames=false]{hyperref}

\usepackage{amssymb,amscd}

\usepackage{mathtools}
\usepackage{manfnt}

\usepackage{sseq}

\usepackage[all]{xy}
\newdir{ >}{{}*!/-8pt/@{>}}

\usepackage[scaled=1.1]{rsfso}
\usepackage{mathrsfs}

\usepackage[numeric,lite]{amsrefs}

\usepackage{pigpen}
\def\PO{\text{\pigpenfont R}}

\renewcommand{\thefootnote}{\fnsymbol{footnote}}
\long\def\symbolfootnote[#1]#2{\begingroup%
\def\thefootnote{\fnsymbol{footnote}}\footnote[#1]{#2}\endgroup}

\newtheorem{thm}{Theorem}[section]
\newtheorem{lem}[thm]{Lemma}
\newtheorem{prop}[thm]{Proposition}
\newtheorem{cor}[thm]{Corollary}

\theoremstyle{definition}
\newtheorem{rem}[thm]{Remark}

\newtheorem{conj}[thm]{Conjecture}

\numberwithin{equation}{section}
\numberwithin{figure}{section}

\def\:{\colon}
\def\.{\cdot}

\def\<{\left\langle}
\def\>{\right\rangle}
\def\({\left(}
\def\){\right)}
\def\ph#1{\phantom{#1}}
\def\epsilon{\varepsilon}
\def\phi{\varphi}

\def\leq{\leqslant}
\def\geq{\geqslant}

\def\Lra{\Longrightarrow}

\def\tilde#1{\widetilde{#1}}
\def\iso{\cong}

\DeclareMathOperator{\Id}{Id}

\DeclareMathOperator{\im}{im}

\def\F{\mathbb{F}}

\def\k{\Bbbk}

\def\N{\mathbb{N}}

\def\Z{\mathbb{Z}}

\DeclareMathOperator{\codim}{codim}

\DeclareMathOperator{\Comod}{Comod}

\DeclareMathOperator{\Ext}{Ext}

\DeclareMathOperator{\Hom}{Hom}
\DeclareMathOperator{\indet}{indet}

\DeclareMathOperator{\TAQ}{TAQ}
\DeclareMathOperator{\Tor}{Tor}

\def\O{\mathrm{O}}
\def\SO{\mathrm{SO}}
\def\Spin{\mathrm{Spin}}
\def\Spinc{\Spin^{\mathrm{c}}}
\def\String{\mathrm{String}}
\def\tmf{\mathrm{tmf}}
\def\U{\mathrm{U}}

\DeclareMathOperator{\exc}{exc}

\DeclareMathOperator{\Sq}{Sq}
\DeclareMathOperator{\q}{q}
\def\dlQ{\mathrm{Q}}
\def\tdlQ{\tilde{\dlQ}}
\DeclareMathOperator{\quo}{\pi}
\def\tpsi{\tilde{\psi}}
\def\Einfty{$\mathcal{E}_\infty$ }
\def\jc{j^\mathrm{c}}

\title[\boldmath\Einfty ring spectra and elements of
Hopf invariant~$1$]
{\boldmath\Einfty ring spectra and elements of Hopf
invariant~$1$}
\author{Andrew Baker}
\date{\version}
\address{
School of Mathematics \& Statistics,
University of Glasgow, Glasgow G12~8QW, Scotland.}
\email{a.baker@maths.gla.ac.uk}
\urladdr{http://www.maths.gla.ac.uk/$\sim$ajb}
\thanks{The author would like to thank the following
for helpful comments over many years: Tilmann Bauer,
Mark Behrens, Irina Bobkova, Bob Bruner, Andr\'e
Henriques, Mike Hill, Rolf Hoyer, Peter Landweber,
Arunas Liulevicius, Peter May, David Pengelley, John
Rognes and Markus Szymik. Part of the work in this
paper was carried out while the author was a participant
in the Hausdorff Trimester Program
\emph{Homotopy theory, manifolds, and field theories}
during July and August~2015 and he would like to
acknowledge the support of the Hausdorff Research
Institute for Mathematics. \\[10pt]
This paper is dedicated to the memory of
\textbf{Sam Gitler} and the final version appears
in a memorial volume of the \emph{Bolet\'in 
de la Sociedad Matem\'atica Mexicana},~\textbf{23} 
(2017), 195--231.
}

\begin{document}

\begin{abstract}
The $2$-primary Hopf invariant~$1$ elements
in the stable homotopy groups of spheres form
the most accessible family of elements. In this
paper we explore some properties of the \Einfty
ring spectra obtained from certain iterated
mapping cones by applying the free algebra
functor. In fact, these are equivalent to Thom
spectra over infinite loop spaces related to
the classifying spaces $B\SO,\,B\Spin,\,B\String$.

We show that the homology of these Thom spectra
are all extended comodule algebras of the form
$\mathcal{A}_*\square_{\mathcal{A}(r)_*}P_*$
over the dual Steenrod algebra $\mathcal{A}_*$
with $\mathcal{A}_*\square_{\mathcal{A}(r)_*}\F_2$
as an algebra retract. This suggests that these
spectra might be wedges of module spectra over
the ring spectra $H\Z$, $k\O$ or $\tmf$, however
apart from the first case, we have no concrete
results on this.
\end{abstract}

\maketitle


\section*{Introduction}

The $2$-primary Hopf invariant~$1$ elements in the stable
homotopy groups of spheres form the most accessible family
of elements. In this paper we explore some properties of
the \Einfty ring spectra obtained from certain iterated
mapping cones by applying the free algebra functor. In
fact, these are equivalent to Thom spectra over infinite
loop spaces related to the classifying spaces $B\SO$,
$B\Spin$ and $B\String$.

We show that the homology of these Thom spectra are all
extended comodule algebras of the form
$\mathcal{A}_*\square_{\mathcal{A}(r)_*}P_*$ over the
dual Steenrod algebra $\mathcal{A}_*$ with
$\mathcal{A}_*\square_{\mathcal{A}(r)_*}\F_2$ as a comodule
algebra retract. This suggests that these spectra might
be wedges of module spectra over the ring spectra $H\Z$,
$k\O$ or $\tmf$, however apart from the first case, we
have no concrete results on this.

Our results and methods of proof owe much to work of
Arunas Liulevicius~\cites{AL:NotesThomSpec,AL:HomComod}
and David Pengelley~\cites{DJP:LondonConf,DJP:MSOMSU,DJP:MO<8>},
and are also related to work of Tony Bahri and Mark
Mahowald~\cite{AB&MM:MO<8>} (indeed there are analogues
of our results for $\mathcal{E}_2$ Thom spectra of the
kind they discuss). However we use some additional
ingredients: in particular we make use of formulae for
the interaction between the $\mathcal{A}_*$-coaction and
the Dyer-Lashof operations in the homology of an \Einfty
ring spectrum described in~\cite{Nishida}. We also take
a slightly different approach to identifying when the
homology of a ring spectrum is a cotensor product of the
dual Steenrod algebra $\mathcal{A}_*$ over a finite quotient
Hopf algebra $\mathcal{A}(n)_*$, making use the fact that
the dual Steenrod algebra is an extended
$\mathcal{A}(n)_*$-comodule; in turn this is a consequence
of Margolis' $P$-algebra property of the Steenrod algebra
$\mathcal{A}^*$.

We remark that the finite complexes of Section~\ref{sec:IterMappingCones}
also appear in the recent preprint by Behrens, Stapleton,
Ormsby \& Stojanoska~\cite{MB-KO-NS-VS:tmf*tmf}: each is
the first of a sequence of generalised integral Brown-Gitler
spectra associated with $H\Z$, $k\O$ and $\tmf$,
see~\cite{MB-KO-NS-VS:tmf*tmf}*{section~2.1}
and~\cites{PG-JDS-MM:GenBGit,FC-DD-PG-MM:IntBGit,SMB:bo^tmf}.
We understand that Bob Bruner and John Rognes have also
considered such spectra.

\tableofcontents

\bigskip
\noindent
\textbf{Conventions:}
We will work $2$-locally throughout this paper, thus
all simply connected spaces and spectra will be assumed
localised at the prime~$2$, and $\mathscr{M}_S$ will
denote the the category of $S$-modules where $S$ is
the $2$-local sphere spectrum as considered in~\cite{EKMM}.
We will write $S^0$ for a chosen cofibrant replacement
for the $S$-module $S$ and $S^n=\Sigma^nS^0$. When
discussing CW skeleta of a space $X$ we will always
assume that we have chosen minimal CW models in the
sense of~\cite{AJB&JPM} so that cells correspond to
a basis of $H_*(X)=H_*(X;\F_2)$.

\noindent
\textbf{Notation:}
When working with cell complexes (of spaces or spectra)
we will often indicate the mapping cone of a coextension
$\tilde{g}$ of a map $g\:S^n\to S^k$ by writing
$X\cup_fe^k\cup_g e^{n+1}$.

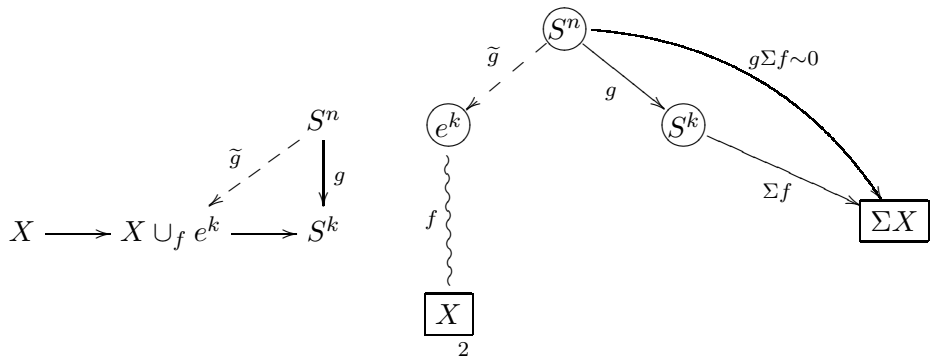
\[
\xymatrix{
&& \\
&& \ar@{-->}[dl]_{\tilde{g}}S^n\ar[d]^g \\
X \ar[r] & X\cup_fe^k\ar[r] & S^k
}
\qquad
\xymatrix@R=0.7cm{
& *+[o][F]{S^n}\ar[dr]_{g}\ar@{-->}[dl]_{\tilde{g}}
     \ar@/^20pt/[ddrrr]^{g\Sigma f\sim0} && & \\
*+[o][F]{e^k}\ar@{~}[dd]_{f}&&*+[o][F]{S^k}\ar[drr]_{\Sigma f}
                              && \\
&&&& *+[F]{\Sigma X}  \\
 *+[F]{X} && &&
}
\]
Of course this notation is ambiguous, but nevertheless
suggestive. When working stably with spectra we will
often write $h\:S^{n+r}\to S^{k+r}$ for the suspension
$\Sigma^rh$ of a map $h\:S^n\to S^k$. We will also often
identify stable homotopy classes with representing
elements.

\section{Iterated mapping cones built with elements
of Hopf invariant~$1$}\label{sec:IterMappingCones}

The results of this section can be proved by homotopy
theory calculations using basic facts about the elements
of Hopf invariant~$1$ in the homotopy groups of the
sphere spectrum $S^0$,
\[
2\in\pi_0(S^0),
\quad
\eta\in\pi_1(S^0),
\quad
\nu\in\pi_3(S^0),
\quad
\sigma\in\pi_7(S^0).
\]
In particular the following identities are well
known, for example see~\cite{DCR:GreenBook}*{figure~A3.1a}:
\begin{equation}
2\eta = \eta\nu = \nu\sigma = 0.
\end{equation}

Although the next result is probably well known,
we outline some details of the constructions of
such spectra, and in particular describe their
homology as $\mathcal{A}_*$-comodules. Later
we will produce naturally occurring examples of
such spectra, but we feel it worthwhile
discussing their construction from a homotopy
theoretic point of view first. We do not address
the question of uniqueness, but it seems possible
that they are unique up to equivalence.

\begin{prop}\label{prop:Iterated}
The following CW spectra exist:
\[
S^0\cup_\eta e^2\cup_2e^3,\quad
S^0\cup_\nu e^4\cup_\eta e^6\cup_2e^7,\quad
S^0\cup_\sigma e^8\cup_\nu e^{12}\cup_\eta e^{14}\cup_2e^{15}.
\]
\end{prop}

\begin{proof}[Sketch of proof]
In each of the iterated mapping cones below, we will denote
the homology generator corresponding to the unique cell in
dimension~$n$ by~$x_n$.

The case of $S^0\cup_\eta\cup_2e^3$ is obvious.

Consider the mapping cone of $\nu$, $C_\nu=S^0\cup_\nu e^4$.
As $\nu\eta=0$, there is a factorisation of $\eta$ on the
$4$-sphere through $C_\nu$.
\[
\xymatrix{
&&& S^5\ar[d]_{2} & \\
&&& \ar@{.>}[dl]_{\tilde{\eta x_4}}S^5\ar[d]_{\eta}\ar[dr]^{\nu\eta=0} & \\
S^3\ar[r]_\nu & S^0\ar[r] & C_\nu\ar[r] & S^4\ar[r]_{\nu}\ar[r] & S^1 \\
}
\]
Also, $2\eta=0$ and $\pi_5(S^0)=0$, hence $2\tilde{\eta x_4}=0$.
A cobar representative for $\tilde{\eta x_4}$ in the classical
Adams $\mathrm{E}_2$-term is
\[
[\zeta_1^2\otimes x_4 + \zeta_2^2\otimes x_0]
  \in \Ext^{1,6}_{\mathcal{A}_*}(\F_2,H_*(C_\nu)).
\]
We can form the mapping cone $C_{\tilde{\eta x_4}}=C_\nu\cup_{\tilde{\eta x_4}}e^6$
and since $2\tilde{\eta x_4}=0$, there is a factorisation of~$2$
on the $6$-sphere through $C_{\tilde{\eta x_4}}$.
\[
\xymatrix{
&&&\ar@{.>}[dl]_{\tilde{2x_6}}\ar[d]_{2}S^6\ar[dr]^{2\tilde{\eta x_4}=0}& \\
S^5\ar[r]_{\phi_1} & C_\nu\ar[r] & C_{\tilde{\eta x_4}}\ar[r]
              & S^6\ar[r]_{\tilde{\eta x_4}} & \Sigma C_\nu \\
}
\]
A cobar representative of $\tilde{2x_6}$ is
\[
[\zeta_1\otimes x_6 + \zeta_2\otimes x_4 + \zeta_3\otimes x_0]
 \in \Ext^{1,7}_{\mathcal{A}_*}(\F_2,H_*(C_{\tilde{\eta x_4}})).
\]

Consider the mapping cone of $\sigma$, $C_\sigma=S^0\cup_\sigma e^8$.
As $\sigma\nu=0$, there is a factorisation of $\nu$ on the $8$-cell
through $C_\sigma$.
\[
\xymatrix{
&&& S^{12}\ar[d]_{\eta} & \\
&&& \ar@{.>}[dl]_{\tilde{\nu x_8}}S^{11}\ar[d]_{\nu}\ar[dr]^{\sigma\nu=0} & \\
S^7\ar[r]_\sigma & S^0\ar[r] & C_\sigma\ar[r] & S^8\ar[r]_{\sigma}\ar[r] & S^1 \\
}
\]
Also, $\nu\eta=0$ and $\pi_{12}(S^0)=0=\pi_{13}(S^0)$,
hence $\eta(\tilde{\nu x_8})=0$.

As $\Ext^{1,12}_{\mathcal{A}_*}(\F_2,H_*(S^0))=0$, the
element
\[
[\zeta_1^4\otimes x_8 + \zeta_2^4\otimes x_0]
  \in \Ext^{1,12}_{\mathcal{A}_*}(\F_2,H_*(C_\sigma))
\]
is a cobar representative for $\tilde{\nu x_8}$.

We can form the mapping cone
$C_{\tilde{\nu x_8}}=C_\sigma\cup_{\tilde{\nu x_8}}e^{12}$
and since $\eta\tilde{\nu x_8}=0$, there is a factorisation
of $\eta$ on the $12$-sphere through $C_{\tilde{\nu x_8}}$.
\[
\xymatrix{
&&&S^{13}\ar[d]_{2}& \\
&&&\ar@{.>}[dl]_{\tilde{\eta x_{12}}}\ar[d]_{\eta}S^{13}\ar[dr]^{\eta(\tilde{\nu x_8})=0}& \\
S^{11}\ar[r]_{\tilde{\nu x_8}} & C_\sigma\ar[r] & C_{\tilde{\nu x_8}}\ar[r]
            & S^{12}\ar[r]_{\tilde{\nu x_8}} & \Sigma C_\sigma \\
}
\]
As part of the long exact sequence for the homotopy
of mapping cone we have the exact sequence
\[
\pi_{13}(S^7) \xrightarrow{\;\sigma\;}\pi_{13}(S^0)
     \xrightarrow{\;\ph{\sigma}\;}\pi_{13}(C_\sigma)
     \xrightarrow{\;\ph{\sigma}\;}\pi_{13}(S^8)
\]
we have $\pi_{13}(S^0)=0=\pi_{13}(S^8)$, so $\pi_{13}(C_\sigma)=0$.
Therefore $2(\tilde{\eta x_{12}})=0$ and we can factorise
$2$ on the $14$-sphere through the mapping cone of
$\tilde{\eta x_{12}}$, $C_{\tilde{\eta x_{12}}}$.
\[
\xymatrix{
&&&\ar@{.>}[dl]_{\tilde{2 x_{14}}}S^{14}\ar[d]_{2}&   \\
S^{13}\ar[r]_{\tilde{\eta x_{12}}} & C_{\tilde{\nu x_8}}\ar[r]
   & C_{\tilde{\eta x_{12}}}\ar[r] & S^{14}\ar[r]_{\tilde{\eta x_{12}}}
   & \Sigma C_{\tilde{\nu x_8}}
}
\]
A cobar representative of $\tilde{2x_{14}}$ is
\[
[\zeta_1\otimes x_{14} + \zeta_2\otimes x_{12}
             + \zeta_3\otimes x_8 + \zeta_4\otimes x_0]
\in\Ext^{1,15}_{\mathcal{A}_*}(\F_2,H_*(C_{\tilde{\eta x_{12}}})).
\]
The homology of the mapping cone $C_{\tilde{2x_{14}}}$
has a basis $x_0,x_8,x_{12},x_{14},x_{15}$, with coaction
given by
\begin{subequations}\label{subeqn:tmf-coaction}
\begin{align}
\psi x_8 &= \zeta_1^8\otimes1 + 1\otimes x_8, \\
\psi x_{12} &=
\zeta_2^4\otimes 1 + \zeta_1^4\otimes x_8 + 1\otimes x_{12},  \\
\psi x_{14} &=
\zeta_3^2\otimes 1 + \zeta_2^2\otimes x_8 + \zeta_1^2\otimes x_{12}
           + 1\otimes x_{14},  \\
\psi x_{15} &=
\zeta_4\otimes 1 + \zeta_3\otimes x_8+ \zeta_2\otimes x_{12}
           + 1\otimes x_{15}.
\end{align}
\end{subequations}

These calculations show that CW spectra of the stated
forms do indeed exist.
\end{proof}
\begin{rem}\label{rem:MinAtom}
The spectra of Proposition~\ref{prop:Iterated} are all
minimal atomic in the sense of~\cite{AJB&JPM}; this
follows from the fact that in each case the mod~$2$
cohomology is a cyclic $\mathcal{A}^*$-module.
\end{rem}

\section{Some \Einfty Thom spectra}
\label{sec:EinftyThomSpectra}

Consider the three infinite loop spaces
$B\SO = B\O\langle2\rangle$, $B\Spin = B\O\langle4\rangle$
and $B\String = B\O\langle8\rangle$. The $3$-skeleton of
$B\SO$ is
\[
B\SO^{[3]} = B\O\langle2\rangle^{[3]} = S^2\cup_2e^3
\]
since $\Sq^1w_2=w_3$. Similarly, the $7$-skeleton
of $B\Spin$ is
\[
B\Spin^{[7]} = B\O\langle4\rangle^{[7]} = S^4\cup_\eta e^6\cup_2e^7
\]
since $\Sq^2w_4=w_6$ and $\Sq^1w_6=w_7$. Finally,
the $15$-skeleton of $B\String$ is
\[
B\String^{[15]} = B\O\langle8\rangle^{[15]}
= S^8\cup_\nu e^{12}\cup_\eta e^{14}\cup_2 e^{15}
\]
since $\Sq^4w_8=w_{12}$, $\Sq^2w_{12}=w_{14}$ and
$\Sq^1w_{14}=w_{15}$.

The skeletal inclusion maps induce (virtual) bundles
whose Thom spectra are themselves skeleta of the
universal Thom spectra $M\SO$, $M\Spin$ and $M\String$.
Routine calculations with Steenrod operations and the
Wu formulae show that
\begin{align*}
M\SO^{[3]}  = M\O\langle2\rangle^{[3]}
           &= S^0\cup_\eta e^2\cup_2 e^3, \\
M\Spin^{[7]} = M\O\langle4\rangle^{[7]}
            &= S^0\cup_\nu e^4\cup_\eta e^6\cup_2 e^7, \\
M\String^{[15]} = M\O\langle8\rangle^{[15]} &=
S^0\cup_\sigma e^8\cup_\nu e^{12}\cup_\eta e^{14}\cup_2 e^{15}.
\end{align*}
Thus these Thom spectra are examples of `iterated
Thom complexes' similar in spirit to those discussed
in~\cite{4A}.

Each skeletal inclusion factors uniquely through
an infinite loop map~$j_r$,
\[
\xymatrix{
B\O\langle2^r\rangle^{[2^{r+1}-1]}\ar[dr]\ar[rr] && B\O\langle2^r\rangle \\
& \dlQ B\O\langle2^r\rangle^{[2^{r+1}-1]}\ar[ur]_{j_1} & \\
}
\]
where $\dlQ=\Omega^\infty\Sigma^\infty$ is the free
infinite loop space functor. We can also form the
associated Thom spectrum $Mj_r$ which is an \Einfty
ring spectrum admitting an \Einfty morphism
$Mj_r\to M\O\langle2^r\rangle$ factoring the corresponding
skeletal inclusion.

Using the algebra of Appendix~\ref{sec:HomConnCov}, it is
easy to see that the skeletal inclusions induce monomorphisms
in homology whose images contain the lowest degree generators:
\begin{align*}
1,a_{1,0}^{(1)},a_{3,0}&\in H_*(M\SO), \\
1,a_{1,0}^{(2)},a_{3,0}^{(1)},a_{7,0}&\in H_*(M\Spin), \\
1,a_{1,0}^{(3)},a_{3,0}^{(2)},a_{7,0}^{(1)},a_{15,0}&\in H_*(M\String).
\end{align*}
Each of the natural orientations $M\O\langle n\rangle\to H\F_2$
above induces an algebra homomorphism
$H_*(M\O\langle n\rangle)\to\mathcal{A}_*$ for which
\[
a_{1,0}^{(r)}\mapsto\zeta_1^{2^r},
\quad a_{3,0}^{(r)}\mapsto\zeta_2^{2^r},
\quad
a_{7,0}^{(r)}\mapsto\zeta_3^{2^r},
\quad
a_{15,0}^{(r)}\mapsto\zeta_4^{2^r}.
\]

We also note that the skeleta can be identified with
skeleta of $H\Z$, $k\O$ and $\tmf$, namely there are
orientations inducing weak equivalences
\begin{equation}\label{eq:SkeletalOrientations}
M\O\langle2\rangle^{[3]} \xrightarrow{\simeq} H\Z^{[3]},
\quad
M\O\langle4\rangle^{[7]} \xrightarrow{\simeq} k\O^{[7]},
\quad
M\O\langle8\rangle^{[15]} \xrightarrow{\simeq} \tmf^{[15]}.
\end{equation}
The first two are induced from well known orientations,
while the third relies on unpublished work of Ando,
Hopkins \& Rezk~\cite{AHR:MO<8>->tmf}. Actually such
morphisms can be produced using the \emph{reduced free
commutative $S$-algebra} functor $\tilde{\mathbb{P}}$
of~\cite{TAQ-PX}, which has a universal property
analogous to that of the usual free functor $\mathbb{P}$
of~\cite{EKMM}.
\begin{prop}\label{prop:RedP-M}
For $r=1,2,3$, the natural map $M\O\langle2^r\rangle^{[2^{r+1}-1]}\to Mj_r$
has a unique extensions to a weak equivalence of \Einfty
ring spectra
\[
\tilde{\mathbb{P}}M\O\langle2^r\rangle^{[2^{r+1}-1]}\xrightarrow{\;\sim\;}Mj_r.
\]
The orientations of \eqref{eq:SkeletalOrientations} induce
morphisms of \Einfty  ring spectra
\[
\tilde{\mathbb{P}}M\O\langle2\rangle^{[3]}\to H\Z,
\quad
\tilde{\mathbb{P}}M\O\langle4\rangle^{[7]}\to k\O,
\quad
\tilde{\mathbb{P}}M\O\langle8\rangle^{[15]}\to\tmf.
\]
\end{prop}
\begin{proof}
The existence of such morphisms depends on the universal
property of $\tilde{\mathbb{P}}$. The proof that those of
the first kind are equivalences depends on a comparison
of the homology rings using Theorem~\ref{thm:H*Mjr} below.
\end{proof}

\begin{rem}\label{rem:SkeletalOrientations}
In fact the weak equivalences of~\eqref{eq:SkeletalOrientations}
extend to weak equivalences
\begin{equation}\label{eq:SkeletalOrientations-2}
Mj_1 \sim H\Z^{[4]},
\quad
Mj_2 \sim k\O^{[8]},
\quad
Mj_3 \sim \tmf^{[16]}.
\end{equation}
\end{rem}

The homology of $Mj_r$ can be determined from that of the
underlying infinite loop space using the Thom isomorphism,
while that for the others depends on a general description
of the homology of $H_*(\tilde{\mathbb{P}}X)$ which can
be found in~\cite{TAQ-PX}.
\begin{thm}\label{thm:H*Mjr}
The homology rings of the Thom spectra $Mj_r$ are given
by
\begin{multline*}
H_*(Mj_1) =
\F_2[\dlQ^Ix_2,\dlQ^Jx_3 :
\text{\rm $I,J$ admissible, $\exc(I)>2$, $\exc(J)>3$}], \\
\shoveleft{H_*(Mj_2) = }  \\
\F_2[\dlQ^Ix_4,\dlQ^Jx_6,\dlQ^Kx_7 :
\text{\rm $I,J,K$ admissible, $\exc(I)>4$, $\exc(J)>6$, $\exc(K)>7$}], \\
\shoveleft{H_*(Mj_3) =
\F_2[\dlQ^Ix_8,\dlQ^Jx_{12},\dlQ^Kx_{14},\dlQ^Lx_{15}
                       : \text{\rm $I,J,K,L$ admissible,} } \\
\text{\rm $\exc(I)>8$, $\exc(J)>12$, $\exc(K)>14$, $\exc(L)>15$}].
\end{multline*}
The \Einfty orientations $Mj_r\to H\F_2$ induce
algebra homomorphisms $H_*(Mj_r)\to\mathcal{A}_*$
which have images
\begin{align*}
\F_2[\zeta_1^2,\zeta_2,\zeta_3,\ldots]
                  &\iso H_*(H\Z), \\
\F_2[\zeta_1^4,\zeta_2^2,\zeta_3,\zeta_4,\ldots]
                  &\iso H_*(k\O), \\
\F_2[\zeta_1^8,\zeta_2^4,\zeta_3^2,\zeta_4,\zeta_5,\ldots]
                  &\iso H_*(\tmf).
\end{align*}
\end{thm}

Recalling Remark~\ref{rem:MinAtom}, we note the
following, where minimal atomic \Einfty ring
spectrum is used in the sense of Hu, Kriz and May,
subsequently developed further in~\cite{BGRtaq}.
\begin{prop}\label{prop:MinAtom}
Each of the \Einfty ring spectra $Mj_r$ $(r=1,2,3)$
is minimal atomic.
\end{prop}
\begin{proof}
In \cite{TAQ-PX} we showed that for $X\in S^0/\mathscr{M}_S$
in the slice category of $S$-modules under a cofibrant
replacement of~$S$,
\[
\Omega_S(\tilde{\mathbb{P}}X) \sim \tilde{\mathbb{P}}X\wedge X/S^0,
\]
hence
\[
\TAQ_*(\tilde{\mathbb{P}}X,S;H) \iso H_*(X/S^0).
\]
For $Mj_r\sim\tilde{\mathbb{P}}M\O\langle2^r\rangle^{[2^{r+1}-1]}$,
this gives
\[
\TAQ_*(Mj_r,S;H) \iso H_*(M\O\langle2^r\rangle^{[2^{r+1}-1]}/S^0).
\]

The $(2^{r+1}-1)$-skeleton for a minimal cell structure
on the spectrum~$Mj_r$ agrees with $M\O\langle2^r\rangle^{[2^{r+1}-1]}$,
and this is a minimal atomic $S$-module as noted in
Remark~\ref{rem:MinAtom}. It follows that the mod~$2$
Hurewicz homomorphism $\pi_*(Mj_r)\to H_*(Mj_r)$ is
trivial in the range $0<*<2^{r+1}$. Hence the $\TAQ$
Hurewicz homomorphism
\[
\pi_*(Mj_r)\to\TAQ_*(Mj_r,S;H)\xrightarrow{\iso} H_*(Mj_r/S^0)
\]
is trivial. Now by~\cite{BGRtaq}*{theorem~3.3},~$Mj_r$
is minimal atomic as claimed.
\end{proof}

\section{Some coalgebra}
\label{sec:Coalgebra}

In this section we review some useful results
on comodules over Hopf algebras. Although most
of this material is standard we state some results
in a precise form suitable for our requirements.
Since writing early versions of this paper we
became aware of work by Hill~\cite{MAH:cyclic}
which uses similar results.

First we recall a standard algebraic result, for
example see~\cite{DJP:MSOMSU}*{lemma~3.1}. We work
with vector spaces over a field $\k$ and will set
$\otimes=\otimes_\k$. There are slight modifications
required for the graded case which we leave the
reader to formulate, however as we work exclusively
in characteristic~$2$, these have no significant
effect in this paper. We refer to the classic paper
of Milnor and Moore~\cite{M&M} for background material
on coalgebra.

Let $A$ be a commutative Hopf algebra over a field~$\k$,
and let $B$ be a quotient Hopf algebra of~$A$. We denote
the product and antipode on $A$ by $\phi_A$ and $\chi$,
and the coaction on a left comodule $D$ by $\psi_D$. We
will identify the cotensor product
$A\square_B\k\subseteq A\otimes\k$ with a subalgebra
of~$A$ under the canonical isomorphism
$A\otimes\k\xrightarrow{\;\iso\;}A$.
\begin{lem}\label{lem:Splitting-Alg}
Let $D$ be a commutative $A$-comodule algebra. Then there
is an isomorphism of $A$-comodule algebras
\begin{equation}\label{eq:Splitting-Alg}
(\phi_A\otimes\Id_{D})\circ(\Id_{A}\otimes\psi_{D})
               \:(A\square_{B}\k)\otimes D
     \xrightarrow{\;\iso\;} A\square_{B}D;
       \quad
a\otimes x \longleftrightarrow \sum_i aa_i\otimes x_i,
\end{equation}
where $\psi_{D}x = \sum_ia_i\otimes x_i$ denotes the
coaction on $x\in D$.
\end{lem}

Here the codomain has the diagonal $A$-comodule structure,
while the domain has the left $A$-comodule structure.

Here is an easily proved generalisation of this result.
\begin{lem}\label{lem:Splitting-Alg-gen}
Let $C$ be a commutative $B$-comodule algebra and let
$D$ be a commutative $A$-comodule algebra, then there
is an isomorphism of $A$-comodule algebras
\begin{equation}\label{eq:Splitting-Alg-gen}
(A\square_{B}C)\otimes D
\xrightarrow{\;\iso\;} A\square_{B}(C\otimes D),
\end{equation}
where the domain has the diagonal left $A$-coaction
and $C\otimes D$ has the diagonal left $B$-coaction.
\end{lem}

Explicitly, on an element
\[
\sum_r u_r\otimes v_r\otimes x
\in (A\square_{B}C)\otimes D \subseteq A\otimes C\otimes D,
\]
the isomorphism has the effect
\[
\sum_r u_r\otimes v_r\otimes w
\longmapsto
\sum_r\sum_i u_ra_i\otimes v_r\otimes w_i,
\]
where $\psi_{D}w = \sum_ia_i\otimes w_i$ as above. Similarly
the inverse is given by
\[
\sum_rb_r\otimes y_r\otimes w_r
\longmapsto
\sum_r\sum_i b_r\chi(a_{r,i})\otimes v_r\otimes w_{r,i}.
\]

Now suppose that $H$ is a finite dimensional Hopf algebra.
If $K$ is a sub-Hopf algebra of $H$, it is well known that
$H$ is a free left or right $K$-module, i.e., $H\iso K\otimes U$
or $H\iso U\otimes K$ for a vector space $U$
(see~\cite{Montgomery}*{theorems~31.1.5 \& 3.3.1}). This
dualises as follows: If $L$ is a quotient Hopf algebra
of $H$, then $H$ is an extended left or right $L$-comodule,
i.e., $H\iso L\otimes V$ or $H\iso V\otimes L$ for
a vector space~$V$; in fact, $V=H\square_L\k$. More
generally, according to Margolis~\cite{Margolis}*{pages~193 \& 240},
if $H$ is a \emph{$P$-algebra} then a result of the
first kind holds for any finite dimensional sub-Hopf
algebra~$K$.

We need to make use of the \emph{finite dual} of a Hopf
algebra~$H$, namely
\[
H^{\mathrm{o}} =
\{ f\in\Hom_\k(H,\k) :
\text{$\exists\,I\lhd H$ such that $\codim I<\infty$
                               and $I\subseteq\ker f$} \}.
\]
Then $H^{\mathrm{o}}$ becomes a Hopf algebra with product
and coproduct obtained from the adjoints of the coproduct
and product of $H$. We will say that $H$ is a \emph{$P$-coalgebra}
if $H^{\mathrm{o}}$ is a $P$-algebra.

\begin{lem}\label{lem:P-coalg}
Suppose that $A$ is a commutative Hopf algebra which is
a $P$-coalgebra. If~$B$ is a finite dimensional quotient
Hopf algebra of $A$, then $A$ is an extended right
\emph{(}or left\emph{)} $B$-comodule, i.e., $A \iso W\otimes B$
\emph{(}or $A \iso B\otimes W$\emph{)} for some vector
space $W$, and in fact $W\iso A\square_B\k$ \emph{(}or
$W\iso\k\square_BA$\emph{)}.
\end{lem}
\begin{cor}\label{cor:P-coalg}
For any right $B$-comodule $L$ or left $B$-comodule $M$,
as vector spaces,
\[
A\square_B M \iso (A\square_B\k)\otimes M,
\quad
L\square_B A \iso L\otimes(\k\square_BA).
\]
These are isomorphisms of left or right $A$-comodules for
suitable comodule structures on the right hand sides.
\end{cor}

To understand the relevant $A$-comodule structure on
$(A\square_B\k)\otimes M$, note that there is an
isomorphism of left $A$-comodules
\[
\xymatrix{
(A\square_B\k)\otimes M
\ar[rr]_(.45){\Id\otimes\psi_M}\ar@/^19pt/[rrrr]^{\iso}
&& (A\square_B\k)\otimes B\otimes M\ar[rr]_(.6){\iso}
&& A\otimes M
}
\]
where the right hand factor is the isomorphism of
Lemma~\ref{lem:P-coalg}.

Crucially for our purposes, for a prime $p$, the
Steenrod algebra $\mathcal{A}^*$ is a $P$-algebra
in the sense of Margolis~\cite{Margolis}, i.e., it
is a union of finite sub-Hopf algebras. When $p=2$,
\[
\mathcal{A}^*=\bigcup_{n\geq0}\mathcal{A}(n)^*,
\]
and it follows from the preceding results that if~$n\geq0$,
$\mathcal{A}^*$ is free as a right or left $\mathcal{A}(n)^*$-module,
see~\cite{Margolis}*{pages~193 \& 240}. Dually,
$(\mathcal{A}_*)^{\mathrm{o}}=\mathcal{A}^*$ and
$\mathcal{A}_*$ is an extended $\mathcal{A}(n)_*$-comodule:
\begin{align}
\mathcal{A}_* &\iso
(\mathcal{A}_*\square_{\mathcal{A}(n)_*}\F_2)\otimes\mathcal{A}(n)_*,
                                     \label{eq:A*extendedright} \\
\mathcal{A}_* &\iso
\mathcal{A}(n)_*\otimes(\F_2\square_{\mathcal{A}(n)_*}\mathcal{A}_*).
                                     \label{eq:A*extendedleft}
\end{align}
Given this, we see that for any left $\mathcal{A}(n)_*$-comodule
$M_*$, as vector spaces
\begin{equation}\label{eq:Acotensor/A(n)}
\mathcal{A}_*\square_{\mathcal{A}(n)_*}M_* \iso
(\mathcal{A}_*\square_{\mathcal{A}(n)_*}\F_2)\otimes M_*.
\end{equation}
In fact this is also an isomorphism of left
$\mathcal{A}_*$-comodules.

Here is an explicit description of isomorphisms
of the type given by Lemma~\ref{lem:P-coalg}.
For $n\geq0$, we will use the function
\[
\mathrm{e}_n\:\N\to\N;
\quad
\mathrm{e}_n(i) =
\begin{dcases*}
2^{n+2-i} & if $1\leq i \leq n+2$, \\
1 & if $i \geq n+3$.
\end{dcases*}
\]
For an natural number $r$, write
\[
r = r'(n,i)\mathrm{e}_n(i) + r''(n,i)
\]
where $0\leq r''(n,i)<\mathrm{e}_n(i)$. We note that
\[
\mathcal{A}_*\square_{\mathcal{A}(n)_*}\F_2
=
\F_2[\zeta_1^{\mathrm{e}_n(1)},\zeta_2^{\mathrm{e}_n(2)},\zeta_3^{\mathrm{e}_n(3)},\ldots]
              \subseteq\mathcal{A}_*,
\]
and
\[
\mathcal{A}(n)_*
= \mathcal{A}_*/\!/(\mathcal{A}_*\square_{\mathcal{A}(n)_*}\F_2)
= \mathcal{A}_*/
(\zeta_1^{\mathrm{e}_n(1)},\zeta_2^{\mathrm{e}_n(2)},\zeta_3^{\mathrm{e}_n(3)},\ldots).
\]
We will indicate elements of $\mathcal{A}(n)_*$ by writing $\|{z}\|$
for the coset of~$z$ which is always chosen to be a sum of monomials
$\zeta_1^{s_1}\zeta_2^{s_2}\cdots\zeta_\ell^{s_\ell}$ with exponents
satisfying $0\leq s_i<\mathrm{e}_n(i)$.
\begin{prop}\label{prop:A*-P-coalg}
For $n\geq0$ there is an isomorphism of right $\mathcal{A}(n)_*$-comodules
\[
\mathcal{A}_*\xrightarrow{\;\iso\;}
(\mathcal{A}_*\square_{\mathcal{A}(n)_*}\F_2)\otimes\mathcal{A}(n)_*
\]
given on basic tensors by
\[
\zeta_1^{r_1}\zeta_2^{r_2}\cdots\zeta_\ell^{r_\ell}
\longleftrightarrow
\zeta_1^{r_1'(n,1)\mathrm{e}_n(1)}
\cdots\zeta_\ell^{r_\ell'(n,\ell)\mathrm{e}_n(\ell)}
\otimes
\left\|\zeta_1^{r_1''(n,1)}
\cdots\zeta_\ell^{r_\ell''(n,\ell)}\right\|.
\]
\end{prop}

We will also use the following result to construct algebraic
maps in lieu of geometric ones. The proof is a straightforward
generalisation of a standard one for the case where~$B=\k$.
\begin{lem}\label{lem:Comodmaps}
Suppose that $M$ is a left $A$-comodule and $N$ is a left
$B$-comodule. Then there is a natural isomorphism
\[
\Comod_B(M,N)\xrightarrow{\;\iso\;} \Comod_A(M,A\square_BN);
\quad
f \mapsto \tilde{f},
\]
where $\tilde{f}$ is the unique factorisation of\/
$(\Id\otimes f)\psi_M$ through $A\square_B N$.
\[
\xymatrix@R=0.6cm{
M\ar[dd]_{\psi_M}\ar[rr]^(.45){\tilde{f}}\ar[ddrr]_{(\Id\otimes f)\psi_M}
&& A\square_B N \ar@{^{(}->}[dd] \\
&& \\
A\otimes M\ar[rr]_{\Id\otimes f} && A\otimes N
}
\]
Furthermore, if $M$ is an $A$-comodule algebra and $N$
is a $B$-comodule algebra, then if $f$ is an algebra
homomorphism, so is $\tilde{f}$.
\end{lem}

As an example of the multiplicative version of this result,
suppose that $M$ is an $A$-comodule algebra which is augmented.
Then there is a composite homomorphism of $B$-comodule algebras
$\alpha\:M \to \k \to N$ giving rise homomorphism of $A$-comodule
algebras
\[
\tilde{\alpha}\:M\to A\square_BN;
\quad
\tilde{\alpha}(x) = a\otimes 1,
\]
where $\psi_M(x) = a\otimes1 + \cdots + 1\otimes x$.

\section{The homology of $Mj_r$ for $r=1,2,3$}\label{sec:H*Mjr}

Now we analyse the the specific cases for $H_*(Mj_r)$
for $r=1,2,3$. Since some of the details differ in
each case we treat these separately. In each case
there is a commutative diagram of commutative
$\mathcal{A}_*$-comodule algebras
\begin{equation}\label{eq:H*Mjr}
\xymatrix@C=.6cm{
&& \\
H_*(Mj_r)\ar@/_15pt/[dr]\ar@/_20pt/[ddr]_(.43){\psi}\ar@/^29pt/@{-->}[rr]^{(\iso)}
& (\mathcal{A}_*\square_{\mathcal{A}(r-1)_*}\F_2)\otimes H_*(Mj_r)
\ar[r]^(.46){\quo}
& (\mathcal{A}_*\square_{\mathcal{A}(r-1)_*}\F_2)\otimes H_*(Mj_r)/I_r \\
& \mathcal{A}_*\square_{\mathcal{A}(r-1)_*} H_*(Mj_r)\ar[u]^{\iso}
\ar[r]^(.46){\quo}\ar@{^{(}->}[d]
&\mathcal{A}_*\square_{\mathcal{A}(r-1)_*} H_*(Mj_r)/I_r\ar@{^{(}->}[d]\ar[u]^{\iso} \\
&\mathcal{A}_*\otimes H_*(Mj_r)
\ar[r]^(.46){\quo}
&\mathcal{A}_*\otimes H_*(Mj_r)/I_r
}
\end{equation}
in which $I_r\lhd H_*(Mj_r)$ is a certain
$\mathcal{A}(r-1)_*$-comodule ideal. In each case
the proof involves showing that the dashed arrow
is an isomorphism.

\subsection{The homology of $Mj_1$}\label{subsec:H*Mj1}

By Theorem~\ref{thm:H*Mjr},
\begin{equation}\label{eq:H*(Mj1)}
H_*(Mj_1) =
\F_2[\dlQ^Ix_2,\dlQ^Jx_3 :
\text{\rm $I,J$ admissible, $\exc(I)>2$, $\exc(J)>3$}],
\end{equation}
where the left $\mathcal{A}_*$-coaction is determined
by
\[
\psi x_2 = 1\otimes x_2 + \zeta_1^2\otimes1,
\quad
\psi x_3 = 1\otimes x_3 + \zeta_1\otimes x_2 + \zeta_2\otimes1.
\]
To calculate the coaction on the other generators
$\dlQ^Ix_2$ and $\dlQ^Jx_3$ we follow~\cite{Nishida}
and use the right coaction
\[
\tpsi\:H_*(Mj_1)\to H_*(Mj_1)\otimes\mathcal{A}_*;
\quad
\tpsi(z) = \sum_i z_i\otimes\chi(\alpha_i),
\]
where $\psi(z) = \sum_i\alpha_i\otimes z_i$ and $\chi$
is the antipode of $\mathcal{A}_*$. So
\[
\tpsi x_2 = x_2\otimes1 + 1\otimes\zeta_1^2,
\quad
\tpsi x_3 = x_3\otimes1 + x_2\otimes\zeta_1 + 1\otimes\xi_2.
\]
In general, if~$z$ has degree~$m$, then
\begin{equation}\label{eq:tpsiQ^rz_n}
\tpsi\dlQ^r z =
\sum_{m\leq k\leq r}
\dlQ^k(\tpsi z)\biggl[\zeta(t)^k\biggr]_{t^r}
=
\sum_{m\leq k\leq r}
\dlQ^k(\tpsi z)\left[\biggl(\frac{\zeta(t)}{t}\biggr)^k\right]_{t^{r-k}}.
\end{equation}

By~\eqref{eq:tpsiQ^rz_n},
\begin{align*}
\tpsi\dlQ^4 x_3 & =
\dlQ^3(x_3\otimes1 + x_2\otimes\zeta_1 + 1\otimes\xi_2)
   \left[\biggl(\frac{\zeta(t)}{t}\biggr)^3\right]_{t} \\
& \ph{aaaaaaaaaaaaaaaaa}
+ \dlQ^4(x_3\otimes1 + x_2\otimes\zeta_1 + 1\otimes\xi_2) \\
& =
x_3^2\otimes\zeta_1 + x_2^2\otimes\zeta_1^3 + 1\otimes\zeta_1\xi_2^2 \\
& \ph{aaaaaaaaaaaaaaaaa}
+ \dlQ^4x_3\otimes1
+ (\dlQ^3x_2\otimes\zeta_1^2 + x_2^2\otimes\dlQ^2\zeta_1)
+ 1\otimes\dlQ^4\xi_2 \\
& =
x_3^2\otimes\zeta_1 + x_2^2\otimes\zeta_1^3
+ 1\otimes\zeta_1\xi_2^2 + \dlQ^4x_3\otimes1 \\
& \ph{aaaaaaaaaaaaaaaaa}
+ \dlQ^3x_2\otimes\zeta_1^2 + x_2^2\otimes\zeta_2
+ 1\otimes(\xi_3+\zeta_1\xi_2^2) \\
& =
(\dlQ^4x_3\otimes1 + x_3^2\otimes\zeta_1 + x_2^2\otimes\xi_2
 + 1\otimes\xi_3) + \dlQ^3x_2\otimes\zeta_1^2.
\end{align*}
We also have
\[
\tpsi\dlQ^3x_2 = \dlQ^3x_2\otimes1,
\quad
\tpsi\dlQ^5x_2 = \dlQ^5x_2\otimes + \dlQ^3x_2\otimes\zeta_1^2.
\]
Combining these we obtain
\begin{align}
\tpsi(\dlQ^4x_3 + \dlQ^5x_2)
&= (\dlQ^4x_3 + \dlQ^5x_2)\otimes1
   + x_3^2\otimes\zeta_1 + x_2^2\otimes\xi_2+1\otimes\xi_3,
\label{eq:Mj1-psix7} \\
\intertext{or equivalently,}
\psi(\dlQ^4x_3 + \dlQ^5x_2)
&= 1\otimes(\dlQ^4x_3 + \dlQ^5x_2) + \zeta_1\otimes x_3^2
    + \zeta_2\otimes x_2^2 + \zeta_3\otimes 1.
\label{eq:Mj1-tpsix7}
\end{align}
We will consider the sequence of elements $X_{1,1}$ and
$X_{1,s}\in H_{2^{s}-1}(Mj_1)$ ($s\geq2$) defined by
\[
X_{1,s} =
\begin{dcases*}
x_2 & if $s=1$, \\
x_3 & if $s=2$, \\
\dlQ^4x_3 + \dlQ^5x_2 & if $s=3$, \\
\dlQ^{(2^{s-1},\ldots,2^4,2^3)}(\dlQ^4x_3 + \dlQ^5x_2)
  = \dlQ^{2^{s-1}}X_{1,s-1} & if $s\geq4$,
\end{dcases*}
\]
where
$\dlQ^{(i_1,i_2,\ldots,i_\ell)}
       =\dlQ^{i_1}\dlQ^{i_2}\cdots\dlQ^{i_\ell}$.
We claim the $X_{1,s}$ have the following right
and left coactions:
\begin{align}
\tpsi X_{1,s} &=
X_{1,s}\otimes1
+
X_{1,s-1}^2\otimes\zeta_1
+\cdots+
X_{1,3}^{2^{s-3}}\otimes\xi_{s-3} \label{eq:X1s-coaction} \\
& \qquad\qquad\qquad\qquad\qquad
+ X_{1,2}^{2^{s-2}}\otimes\xi_{s-2}
+ X_{1,1}^{2^{s-2}}\otimes\xi_{s-1}
+ 1\otimes\xi_{s}, \notag \\
\psi X_{1,s} &=
1\otimes X_{1,s}
+
\zeta_1\otimes X_{1,s-1}^2
+\cdots+
\zeta_{s-3}\otimes X_{1,3}^{2^{s-3}} \label{eq:X1s-leftcoaction} \\
& \qquad\qquad\qquad\qquad\qquad
+ \zeta_{s-2}\otimes X_{1,2}^{2^{s-2}}
+ \zeta_{s-1}\otimes X_{1,1}^{2^{s-2}}
+ \zeta_{s}\otimes 1.  \notag
\end{align}

To prove these, we use induction on~$s$, where the early
cases~$s=1,2,3$ are known already. For the inductive step,
assume that~\eqref{eq:X1s-coaction} holds for some $s\geq3$.
Then
\begin{align*}
\tpsi X_{1,s+1} =& \tpsi\dlQ^{2^s}X_{1,s}
= (\tpsi X_{1,s})^2\zeta_1 + \dlQ^{2^s}(\tpsi X_{1,s}) \\
=&
X_{1,s}^2\otimes\zeta_1
+
X_{1,s-1}^{2^2}\otimes\zeta_1^3
+\cdots+
X_{1,3}^{2^{s-2}}\otimes\xi_{s-3}^2\zeta_1 \\
&\ph{aaaaaaaaaaaaaaaaa} + X_{1,2}^{2^{s-1}}\otimes\xi_{s-2}^2\zeta_1
+ X_{1,1}^{2^{s-1}}\otimes\xi_{s-1}^2\zeta_1
+ 1\otimes\xi_{s}^2\zeta_1 \\
&
+ \dlQ^{2^s}
\biggl(X_{1,s}\otimes1
+
X_{1,s-1}^2\otimes\zeta_1
+\cdots+
X_{1,3}^{2^{s-3}}\otimes\xi_{s-3} \biggr. \\
&\ph{aaaaaaaaaaaaaaaaaaaaaaaaaaaaa}
\biggl. + X_{1,2}^{2^{s-2}}\otimes\xi_{s-2}
+ X_{1,1}^{2^{s-2}}\otimes\xi_{s-1}
+ 1\otimes\xi_{s}\biggr) \\
=& X_{1,s}^2\otimes\zeta_1
+
X_{1,s-1}^{2^2}\otimes\zeta_1^3
+\cdots+
X_{1,3}^{2^{s-2}}\otimes\xi_{s-3}^2\zeta_1
+ X_{1,2}^{2^{s-1}}\otimes\xi_{s-2}^2\zeta_1 \\
&\ph{aaaaaaaaaaaaaaaaaaaaaaaaaaaaa}
+ X_{1,1}^{2^{s-1}}\otimes\xi_{s-1}^2\zeta_1
+ 1\otimes\xi_{s}^2\zeta_1 \\
&\quad
+ \dlQ^{2^s}X_{1,s}\otimes1
+
X_{1,s-1}^{2^2}\otimes\dlQ^{2}\zeta_1
+\cdots+
X_{1,3}^{2^{s-2}}\otimes\dlQ^{2^{s-3}}\xi_{s-3} \\
&\ph{aaaaaaaaaaa}
+ X_{1,2}^{2^{s-1}}\otimes\dlQ^{2^{s-2}}\xi_{s-2}
+ X_{1,1}^{2^{s-1}}\otimes\dlQ^{2^{s-1}}\xi_{s-1}
+ 1\otimes\dlQ^{2^{s}}\xi_{s} \\
=& X_{1,s}^2\otimes\zeta_1
+
X_{1,s-1}^{2^2}\otimes\zeta_1^3
+\cdots+
X_{1,3}^{2^{s-2}}\otimes\xi_{s-3}^2\zeta_1
+ X_{1,2}^{2^{s-1}}\otimes\xi_{s-2}^2\zeta_1 \\
&\ph{aaaaaaaaaaaaaaaaaaaaaaaaaaaaa}
+ X_{1,1}^{2^{s-1}}\otimes\xi_{s-1}^2\zeta_1
+ 1\otimes\xi_{s}^2\zeta_1 \\
&\quad
+ X_{1,s+1}\otimes1
+
X_{1,s-1}^{2^2}\otimes(\xi_{2}+\zeta_1^3)
+\cdots+
X_{1,3}^{2^{s-2}}\otimes(\xi_{s-2}+\xi_{s-3}^2\zeta_1) \\
&\ph{aaaaaaaaa}
+ X_{1,2}^{2^{s-1}}\otimes(\xi_{s-1}+\xi_{s-2}^2\zeta_1)
+ X_{1,1}^{2^{s-1}}\otimes(\xi_{s}+\xi_{s-1}^2\zeta_1)
+ 1\otimes(\xi_{s+1}+\xi_{s}^2\zeta_1)  \\
=& X_{1,s+1}\otimes1 + X_{1,s}^2\otimes\zeta_1
+
X_{1,s-1}^{2^2}\otimes\xi_2
+\cdots+
X_{1,3}^{2^{s-2}}\otimes\xi_{s-2} \\
&\ph{aaaaaaaaaaaaaaaaaaaaaaaaaaaa}
+ X_{1,2}^{2^{s-1}}\otimes\xi_{s-1}
+ X_{1,1}^{2^{s-1}}\otimes\xi_{s}
+ 1\otimes\xi_{s+1},
\end{align*}
giving the result for $s+1$. Here for terms
of form $\dlQ^{|u|+|v|+1}(u\otimes v)$ we
have
\[
\dlQ^{|u|+|v|+1}(u\otimes v)
= \dlQ^{|u|+1}u\otimes\dlQ^{|v|}v
             + \dlQ^{|u|+1}u\otimes\dlQ^{|v|+1}v
= \dlQ^{|u|+1}u\otimes v^2 + u^2\otimes\dlQ^{|v|+1}v
\]
by the Cartan formula and unstable conditions.

Under the homomorphism $\rho\:H_*(Mj_1)\to\mathcal{A}_*$
induced by the orientation $Mj_1\to H\F_2$, we have
\[
\rho(x_2)=\zeta_1^2, \quad \rho(x_3)=\zeta_2,
\quad \rho(X_{1,s})=\zeta_s \;\;\; (s\geq3).
\]
Also
\[
\rho(\dlQ^3 x_2) = \dlQ^3(\rho x_2) = \dlQ^3(\zeta_1^2) = 0,
\]
and for each admissible monomial~$I$,
$\rho(\dlQ^Ix_2)\in\mathcal{A}_*$ is a square.

This shows that the restriction of $\rho$ to the subalgebra
generated by the $X_{1,s}$ is an isomorphism of
$\mathcal{A}_*$-comodule algebras
\[
\F_2[X_{1,s}:s\geq1]\xrightarrow{\;\iso\;}
\mathcal{A}_*\square_{\mathcal{A}(0)_*}\F_2 \subseteq\mathcal{A}_*,
\]
where
\[
\mathcal{A}(0)_* =
\mathcal{A}_*/\!/\F_2[\zeta_1^2,\zeta_2,\zeta_3,\ldots],
\quad
\mathcal{A}_*\square_{\mathcal{A}(0)_*}\F_2 =
\F_2[\zeta_1^2,\zeta_2,\zeta_3,\ldots]\subseteq\mathcal{A}_*.
\]


In the algebra $H_*(Mj_1)$, the regular sequence
$X_{1,s}$ ($s\geq1$) generates an ideal
\[
I_1 = ( X_{1,s} : s\geq1 ) \lhd H_*(Mj_1).
\]
This is not an $\mathcal{A}_*$-subcomodule since for
example,
\[
\psi X_{1,3} = \psi(\dlQ^4x_3 + \dlQ^5x_2)
= (1\otimes X_{1,3} + \zeta_1\otimes X_{1,2}^2
      + \zeta_2\otimes X_{1,1}^2) + \zeta_3\otimes1.
\]
However under the induced $\mathcal{A}(0)_*$-coaction
\[
\psi'\: H_*(Mj_1)\to\mathcal{A}(0)_*\otimes H_*(Mj_1),
\]
the last term becomes trivial, in fact
\[
\psi' X_{1,3}
    = 1\otimes X_{1,3} + \zeta_1\otimes X_{1,2}^2,
\]
where we identify elements of $\mathcal{A}(0)_*$
with representatives in $\mathcal{A}_*$. More
generally, by~\eqref{eq:X1s-leftcoaction}, for
$s\geq2$,
\[
\psi' X_{1,s} = 1\otimes X_{1,s} + \zeta_1\otimes X_{1,s-1}^2.
\]
It follows that $I_1$ is an $\mathcal{A}(0)_*$-invariant
ideal.
\begin{prop}\label{prop:H*Mj1-extended}
There is an isomorphism of commutative $\mathcal{A}_*$-comodule
algebras
\[
H_*(Mj_1) \xrightarrow{\;\iso\;}
\mathcal{A}_*\square_{\mathcal{A}(0)_*} H_*(Mj_1)/I_1.
\]
\end{prop}
\begin{proof}
Taking $r=1$, from~\eqref{eq:H*Mjr} we obtain
a commutative diagram of commutative
$\mathcal{A}_*$-comodule algebras
\[
\xymatrix@C=.7cm{
&& \\
H_*(Mj_1)\ar@/_15pt/[dr]\ar@/_20pt/[ddr]_(.43){\psi}\ar@/^29pt/@{-->}[rr]
& (\mathcal{A}_*\square_{\mathcal{A}(0)_*}\F_2)\otimes H_*(Mj_1)
\ar[r]^(.46){\quo}
& (\mathcal{A}_*\square_{\mathcal{A}(0)_*}\F_2)\otimes H_*(Mj_1)/I_1 \\
& \mathcal{A}_*\square_{\mathcal{A}(0)_*} H_*(Mj_1)\ar[u]^{\iso}
\ar[r]^(.46){\quo}\ar@{^{(}->}[d]
&\mathcal{A}_*\square_{\mathcal{A}(0)_*} H_*(Mj_1)/I_1\ar@{^{(}->}[d]\ar[u]^{\iso} \\
&\mathcal{A}_*\otimes H_*(Mj_1)
\ar[r]^(.46){\quo}
&\mathcal{A}_*\otimes H_*(Mj_1)/I_1
}
\]
and furthermore
\begin{align*}
\psi X_{1,1} &= \zeta_1^2\otimes1 + 1\otimes X_{1,1}, \\
\psi X_{1,2} &=
\zeta_2\otimes1 + \zeta_1\otimes X_{1,1} + 1\otimes X_{1,1}, \\
\psi X_{1,s} &=
\zeta_{s+1}\otimes1 + \cdots + 1\otimes X_{1,s}
\qquad(s\geq3),
\end{align*}
giving
\[
\quo\psi X_{1,1} = \zeta_1^2\otimes1,
\quad
\quo\psi X_{1,2} = \zeta_2\otimes1,
\quad
\quo\psi X_{1,s} =
\zeta_{s+1}\otimes1 + \cdots.
\]
The latter form part of a set of polynomial generators
for the polynomial ring
\[
\mathcal{A}_*\otimes H_*(Mj_1)/I_1\iso
(\mathcal{A}_*\square_{\mathcal{A}(0)_*}\F_2)\otimes H_*(Mj_1)/I_1.
\]
Now a straightforward argument shows that the dashed arrow
is surjective; but as the Poincar\'e series of $H_*(Mj_1)$
and $(\mathcal{A}_*\square_{\mathcal{A}(0)_*}\F_2)\otimes H_*(Mj_1)/I_1$
are equal, it is actually an isomorphism. Therefore
\[
H_*(Mj_1) \iso
\mathcal{A}_*\square_{\mathcal{A}(0)_*} H_*(Mj_1)/I_1.
\qedhere
\]
\end{proof}
\begin{rem}\label{rem:H*Mj1-extended}
For the purposes of proving such a result, we might
as well have set $X_{1,3} = \dlQ^4x_3$ and
\[
X_{1,s} = \dlQ^{2^{s-1}}X_{1,s-1} \quad (s\geq3),
\]
since
\[
\psi'X_{1,3} = 1\otimes X_{1,3} + \zeta_1\otimes x_3^2
\]
and so on. However, the cases of $Mj_2$ and $Mj_3$ will
require modifications similar to the ones we have used
above which give an indication of the methods required.
\end{rem}

We have the following splitting result.
\begin{prop}\label{prop:H*Mj1-splitting}
There is a splitting of $\mathcal{A}_*$-comodule algebras
\[
\xymatrix{
\mathcal{A}_*\square_{\mathcal{A}(0)_*}\F_2\ar[rr]^{\iso}\ar[dr]
&& \mathcal{A}_*\square_{\mathcal{A}(0)_*}\F_2 \\
& H_*(Mj_1)\ar[ur] &
}
\]
where $H_*(Mj_1)\to H_*(H\Z)=\mathcal{A}_*\square_{\mathcal{A}(0)_*}\F_2$
is induced by the \Einfty orientation $Mj_1\to H\Z$.
\end{prop}
\begin{proof}
This is proved using Lemma~\ref{lem:Comodmaps} together with
the trivial $\mathcal{A}(0)_*$-comodule algebra homomorphism
$\mathcal{A}_*\square_{\mathcal{A}(0)_*}\F_2\to H_*(Mj_1)/I_1$.
\end{proof}

\subsection{The homology of $Mj_2$}\label{subsec:H*Mj2}

We have
\begin{multline*}
H_*(Mj_2) =   
\F_2[\dlQ^Ix_4,\dlQ^Jx_6,\dlQ^Kx_7 :
\text{\rm $I,J,K$ admissible, $\exc(I)>4$,
                     $\exc(J)>6$, $\exc(K)>7$}],
\end{multline*}
with right coaction satisfying
\begin{align*}
\tpsi x_4 &=
x_4\otimes1 + 1\otimes\zeta_1^4, \\
\tpsi x_6 &=
x_6\otimes1 + x_4\otimes\zeta_1^2 + 1\otimes\xi_2^2, \\
\tpsi x_7 &=
x_7\otimes1 + x_6\otimes\zeta_1 + x_4\otimes\xi_2 + 1\otimes\xi_3.
\end{align*}
Furthermore,
\begin{align*}
\tpsi\dlQ^8 x_7 &=
x_7^2\otimes\zeta_1 + x_6^2\otimes\zeta_1^3
+ x_4^2\otimes\zeta_1\xi_2^2 + 1\otimes\xi_3^2\zeta_1 \\
& \ph{aaaaaaaaaaaaaaaaaaaaaaa}
+ \dlQ^8(x_7\otimes1 + x_6\otimes\zeta_1 + x_4\otimes\xi_2 + 1\otimes\xi_3) \\
&=
x_7^2\otimes\zeta_1 + x_6^2\otimes\zeta_1^3
+ x_4^2\otimes\zeta_1\xi_2^2 + 1\otimes\xi_3^2\zeta_1
+ \dlQ^8x_7 + \dlQ^7x_6\otimes\zeta_1^2 \\
&\ph{aaaaaaaaaaa}
+ \dlQ^5x_4\otimes\xi_2^2
+ 1\otimes(\xi_4+\zeta_1\xi_3^2) + x_6^2\otimes\zeta_2
+ x_4^2\otimes(\xi_3+\zeta_1\xi_2^2)  \\
&=
(\dlQ^8x_7 + x_7^2\otimes\zeta_1 + x_6^2\otimes\xi_2
                    + x_4^2\otimes\xi_3 + 1\otimes\xi_4)
 + \dlQ^7x_6\otimes\zeta_1^2 + \dlQ^5x_4\otimes\xi_2^2,
\end{align*}
so the left $\mathcal{A}(1)_*$-coproduct
\[
\psi'\: H_*(Mj_2) \to \mathcal{A}(1)_*\otimes H_*(Mj_2)
\]
satisfies
\begin{align*}
\psi'\dlQ^8 x_7 &=
(\dlQ^8x_7 + \zeta_1\otimes x_7^2 +\zeta_2\otimes x_6^2
              + \zeta_3\otimes x_4^2 + \zeta_4\otimes1)
 +\zeta_1^2\otimes\dlQ^7x_6 + \zeta_2^2\otimes\dlQ^5x_4 \\
&=
(\dlQ^8x_7 + \zeta_1\otimes x_7^2 +\zeta_2\otimes x_6^2)
                             +\zeta_1^2\otimes\dlQ^7x_6.
\end{align*}
We also have
\begin{align*}
\psi'\dlQ^9 x_6 &=
1\otimes\dlQ^9 x_6 + \zeta_1^2\otimes\dlQ^7x_6
+ \zeta_1^4\otimes\dlQ^7x_4 + \zeta_2^2\otimes\dlQ^5x_4 \\
&= 1\otimes\dlQ^9 x_6 + \zeta_1^2\otimes\dlQ^7x_6,
\end{align*}
so
\[
\psi'(\dlQ^8 x_7+\dlQ^9 x_6) =
1\otimes\dlQ^8x_7 + \zeta_1\otimes x_7^2 +\zeta_2\otimes x_6^2
\in \mathcal{A}(1)_*\otimes H_*(Mj_2).
\]

Now we define a sequence of elements $X_{2,s}$ ($s\geq1$)
by
\[
X_{2,s} =
\begin{dcases*}
x_4 & if $s=1$, \\
x_6 & if $s=2$, \\
x_7 & if $s=3$, \\
\dlQ^8 x_7+\dlQ^9 x_6 & if $s=4$, \\
\dlQ^{(2^{s-1},\ldots,2^5,2^4)}(\dlQ^8 x_7+\dlQ^9 x_6)
    = \dlQ^{2^{s-1}}X_{2,s-1} & if $s\geq5$.
\end{dcases*}
\]
An inductive calculation shows that for $s\geq4$,
\[
\psi'X_{2,s} =
1\otimes X_{2,s} + \zeta_1\otimes X_{2,s-1}^2 + \zeta_2\otimes X_{2,s-2}^4
\in \mathcal{A}(1)_*\otimes I_2.
\]
So this sequence is regular and generates an $\mathcal{A}(1)_*$-invariant
ideal
\[
I_2 = (X_{2,s} : s\geq1)\lhd H_*(Mj_2).
\]
The next result follows using similar arguments to those
in the proof of Proposition~\ref{prop:H*Mj1-extended}
using the diagram~\eqref{eq:H*Mjr}.
\begin{prop}\label{prop:H*Mj2-extended}
There is an isomorphism of $\mathcal{A}_*$-comodule
algebras
\[
H_*(Mj_2) \xrightarrow{\;\iso\;}
\mathcal{A}_*\square_{\mathcal{A}(1)_*} H_*(Mj_2)/I_2.
\]
\end{prop}

The \Einfty  morphism $Mj_2\to k\O$ induces an
algebra homomorphism
$H_*(Mj_2)\to H_*(k\O)\subseteq\mathcal{A}_*$
under which
\[
X_{2,1}\mapsto \zeta_1^4,\quad X_{2,2}\mapsto \zeta_2^2,
\quad X_{2,s}\mapsto \zeta_s\;\;\; (s\geq3).
\]

We have the following splitting result analogous
to Proposition~\ref{prop:H*Mj1-splitting}.
\begin{prop}\label{prop:H*Mj2-splitting}
There is a splitting of $\mathcal{A}_*$-comodule algebras
\[
\xymatrix{
\mathcal{A}_*\square_{\mathcal{A}(1)_*}\F_2\ar[rr]^{\iso}\ar[dr]
&& \mathcal{A}_*\square_{\mathcal{A}(1)_*}\F_2 \\
& H_*(Mj_2)\ar[ur] &
}
\]
where $H_*(Mj_2)\to H_*(k\O)=\mathcal{A}_*\square_{\mathcal{A}(1)_*}\F_2$
is induced by the \Einfty orientation $Mj_2\to k\O$.
\end{prop}

\subsection{The homology of $Mj_3$}\label{subsec:H*Mj3}

In $H_*(Mj_3)$, consider the regular sequence
\[
X_{3,s} =
\begin{dcases*}
x_8 & if $s=1$, \\
x_{12} & if $s=2$, \\
x_{14} & if $s=3$, \\
x_{15} & if $s=4$, \\
\dlQ^{16}x_{15} + \dlQ^{17}x_{14} + \dlQ^{19}x_{12} & if $s=5$, \\
\dlQ^{(2^{s-1},\ldots,2^6,2^5)}(\dlQ^{16}x_{15} + \dlQ^{17}x_{14} + \dlQ^{19}x_{12})
    = \dlQ^{2^{s-1}}X_{3,s-1} & if $s\geq6$.
\end{dcases*}
\]
We leave the reader to verify that the ideal
\[
I_3 = ( X_{3,s} : s\geq1 )\lhd H_*(Mj_3)
\]
is $\mathcal{A}(2)_*$-invariant. The proof of the
following result is similar to those of
Propositions~\ref{prop:H*Mj1-extended} and~\ref{prop:H*Mj2-extended}
using the diagram~\eqref{eq:H*Mjr}.
\begin{prop}\label{prop:H*Mj3-extended}
There is an isomorphism of $\mathcal{A}_*$-comodule
algebras
\[
H_*(Mj_3) \xrightarrow{\;\iso\;}
\mathcal{A}_*\square_{\mathcal{A}(2)_*} H_*(Mj_3)/I_3.
\]
\end{prop}

The \Einfty morphism $Mj_3\to\tmf$ induces an
algebra homomorphism
$H_*(Mj_3)\to H_*(\tmf)\subseteq\mathcal{A}_*$
under which
\[
X_{3,1}\mapsto \zeta_1^8,\quad X_{3,2}\mapsto \zeta_2^4,
\quad X_{3,3}\mapsto \zeta_3^2,
\quad X_{3,s}\mapsto \zeta_s\;\;\; (s\geq3).
\]

We have the following splitting result analogous
to Propositions~\ref{prop:H*Mj1-splitting}
and~\ref{prop:H*Mj2-splitting}.
\begin{prop}\label{prop:H*Mj3-splitting}
There is a splitting of $\mathcal{A}_*$-comodule
algebras
\[
\xymatrix{
\mathcal{A}_*\square_{\mathcal{A}(2)_*}\F_2\ar[rr]^{\iso}\ar[dr]
&& \mathcal{A}_*\square_{\mathcal{A}(2)_*}\F_2 \\
& H_*(Mj_3)\ar[ur] &
}
\]
where
$H_*(Mj_3)\to H_*(\tmf)=\mathcal{A}_*\square_{\mathcal{A}(2)_*}\F_2$
is induced by the \Einfty orientation $Mj_3\to\tmf$.
\end{prop}

We end this discussion by recording the following result
which was in part motivated by a result of Lawson \&
Naumann~\cite{TL&NN}.
\begin{thm}\label{thm:tmf->kO}
There is a morphism of \Einfty ring spectra $Mj_3\to k\O$
which induces an epimorphism
\[
H_*(Mj_3) \twoheadrightarrow
\F_2[\zeta_1^8,\zeta_2^4,\zeta_3^2,\zeta_4,\zeta_5,\ldots]
\subseteq\F_2[\zeta_1^4,\zeta_2^2,\zeta_3,\zeta_4,\zeta_5,\ldots]
\iso H_*(k\O)
\]
on $H_*(-)$ and an epimorphism $\pi_k(Mj_3)\to\pi_k(k\O)$
for $k\neq 4$.
\end{thm}
\begin{proof}
We will use the fact that $Mj_3\sim\tilde{\mathbb{P}}\tmf^{[15]}$
and show the existence of a suitable \Einfty morphism
$\tilde{\mathbb{P}}\tmf^{[15]}\to k\O$.

We first require a map $\tmf^{[15]}\to k\O$ extending
the unit map $S^0\to k\O$. The existence of maps can
be shown using classical obstruction theory since the
successive obstructions lie in the groups
$H^8(\tmf^{[15]};\pi_7(k\O))$,
$H^{12}(\tmf^{[15]};\pi_{11}(k\O))$,
$H^{14}(\tmf^{[15]};\pi_{13}(k\O))$
and
$H^{15}(\tmf^{[15]};\pi_{14}(k\O))$,
all of which are trivial. For definiteness,
choose such a map $\theta\:\tmf^{[15]}\to k\O$.

Let us examine the induced $\mathcal{A}_*$-comodule
homomorphism
$\theta_*\:H_*(\tmf^{[15]})\to H_*(k\O)\subseteq\mathcal{A}_*$.
By Lemma~\ref{lem:Comodmaps} we have
\begin{align*}
\Comod_{\mathcal{A}_*}(H_*(\tmf^{[15]}),H_*(k\O))
&\iso
\Comod_{\mathcal{A}_*}(H_*(\tmf^{[15]}),
     \mathcal{A}_*\square_{\mathcal{A}(1)_*}\F_2) \\
&\iso
\Comod_{\mathcal{A}(1)_*}(H_*(\tmf^{[15]}),\F_2)\iso\F_2,
\end{align*}
so $\theta_*$ is a uniquely determined. Recall
the formulae for the coaction on $H_*(\tmf^{[15]})$
given in~\eqref{subeqn:tmf-coaction}, we find
that
\[
\theta_*(x_8) = \zeta_1^8,\quad
\theta_*(x_{12}) = \zeta_2^4,\quad
\theta_*(x_{14}) = \zeta_3^2,\quad
\theta_*(x_{15}) = \zeta_4.
\]

There is a unique extension of $\theta$ to a morphism
of \Einfty ring spectra
$\tilde{\theta}\:\tilde{\mathbb{P}}\tmf^{[15]}\to k\O$.
The homology of $\tilde{\mathbb{P}}\tmf^{[15]}$ is
given in Theorem~\ref{thm:H*Mjr}, and for $s\geq5$
\[
\tilde{\theta}_*(X_{3,s})
= \dlQ^{(2^{s-1},\ldots,2^6,2^5)}(\theta_*(x_{15}))
= \dlQ^{(2^{s-1},\ldots,2^6,2^5)}(\zeta_4)
= \zeta_s.
\]
It follows that
\[
\im\tilde{\theta}_*
= \F_2[\zeta_1^8,\zeta_2^4,\zeta_3^2,\zeta_4,\zeta_5,\ldots]
\iso H_*(\tmf).
\]

To prove the result about homotopy groups, we show
first that $\theta_*\:\pi_k(\tmf^{[15]})\to\pi_k(k\O)$
is surjective when $k=8,9,10,12$. We will use arguments
about some Toda brackets in $\pi_*(\tmf^{[15]})$ and
$\pi_*(k\O)$; similar results were used
in~\cite{AJB&JPM}*{section~7}. Given an $S$-module~$X$,
we can define Toda brackets of the form
$\langle\alpha,\beta,\gamma\rangle\subseteq\pi_{a+b+c+1}(X)$,
where $\alpha\in\pi_a(S)$, $\beta\in\pi_b(S)$ and $\gamma\in\pi_c(X)$
satisfy $\alpha\beta=0$ in $\pi_{a+b}(S)$ and $\beta\gamma=0$
in $\pi_{b+c}(X)$. The indeterminacy here is as usual
\[
\indet\langle \alpha,\beta,\gamma\rangle =
\alpha\pi_{b+c+1}(X) + \pi_{a+b+1}(S)\gamma
              \subseteq\pi_{a+b+c+1}(X).
\]

The case $k=8$ follows from the well known facts
that the Toda brackets
$\langle 16,\sigma,1\rangle\subseteq\pi_8(\tmf)$
and $\langle 16,\sigma,1\rangle\subseteq\pi_8(k\O)$
contain generators $c'_4\in\pi_8(\tmf)\iso\pi_8(\tmf^{[15]})$
and $w\in\pi_8(k\O)$ respectively. Naturality
shows that $\theta_*\:\pi_8(\tmf^{[15]})\to\pi_8(k\O)$
is surjective.

For the cases $k=9,10$ we can use mutiplication
by $\eta$ and $\eta^2$ in $\pi_*(\tmf)^{[15]}$
and $\pi_*(k\O)$ to see that
$\theta_*\:\pi_k(\tmf)^{[15]}\to\pi_k(k\O)$ is
surjective in these cases.

For $k=12$ we need to know the classical result
$\nu w=0$ well as $\nu c'_4=0$; the latter can
be read off of the Adams spectral sequence diagrams
in~\cite{TMF}*{chapter~13}. Given these facts,
it follows that the Toda brackets
$\langle 8,\nu,c'_4\rangle\subseteq\pi_{12}(\tmf)
         \iso\pi_{12}(\tmf^{[15]})$
and $\langle 8,\nu,w\rangle\subseteq\pi_{12}(k\O)$
contain generators and naturality shows that
$\theta_*\:\pi_{12}(\tmf^{[15]})\to\pi_{12}(k\O)$
is surjective.

To finish our argument, we know that when $k=8,9,10,12$
the composition
\[
\xymatrix{
\pi_k(\tmf^{[15]}) \ar[r]\ar@/^20pt/[rr]^{\theta_*}
& \pi_k(\tilde{\mathbb{P}}\tmf^{[15]})\ar[r]_(.6){\tilde{\theta}_*}
& \pi_k(k\O)
}
\]
is surjective. Using multiplication by the image
of $c'_4$ in $\pi_*(\tilde{\mathbb{P}}\tmf^{[15]})$
it is straightforward to show that
$\theta_*\:\pi_k(\tmf^{[15]})\to\pi_k(k\O)$ is
surjective for all $k>4$.
\end{proof}

In~\cite{TL&NN}, Lawson and Naumann have shown the
existence of an \Einfty map $\tmf\to k\O$ whose
restriction to $\tmf^{[15]}$ could be used in the
proof above. However, our argument does not assume
the prior existence of such a map and seems more
elementary. Indeed, our result suggests the possibility
of a more direct approach to building an \Einfty
morphism $\tmf\to k\O$ in comparison with the approach
of Lawson and Naumann: it would suffice to show
that the map $\mathcal{I}\to k\O$ from the homotopy
fibre $\mathcal{I}$ of the \Einfty morphism
$\tilde{\mathbb{P}}\tmf^{[15]}\to k\O$ was null
homotopic, so there is an \Einfty morphism
$\tmf\to k\O$ making the following diagram homotopy
commutative.
\[
\xymatrix{
\mathcal{I}\ar[d]\ar[dr] & \\
\tilde{\mathbb{P}}\tmf^{[15]}\ar[d]\ar[r]_(.57){\tilde{\theta}} & k\O \\
\tmf\ar@/_11pt/@{.>}[ur]
}
\]
To date we have been unable to make this approach
work.

\section{Some other examples}\label{sec:OtherExamp}

Our approach to proving algebraic splittings of
the homology of \Einfty Thom spectra can be used
to rederive many known results for classical
examples such as $M\O$, $M\SO$, $M\SO$, $M\Spin$,
$M\String=M\O\langle8\rangle$ and $M\U$. We can
also obtain some other new examples with these
methods.

\subsection{An example related to $k\U$}\label{subsec:Examp-kU}
Our first example is based on similar ideas to
those used to construct the spectra $Mj_r$, but
using $\Spinc$. The low dimensional homology of
$B\Spinc$ can be read off from Theorem~\ref{thm:Spinc}
and Remark~\ref{rem:Spinc}. Passing to the Thom
spectrum over the $7$-skeleton $(B\Spinc)^{[7]}$
we have for its homology
\[
H_*((M\Spinc)^{[7]}) =
\F_2\{1,a_{1,0}^{(1)},a_{1,1}^{(1)},(a_{1,0}^{(1)})^2,
                      a_{3,0}^{(1)},a_{7,0}\}.
\]
For our purposes, the fact that there are two
$4$-cells is problematic, so we instead restrict
to a smaller complex. The map
$B\Spin^{[7]}\to B\Spin^c$ induces an epimorphism
in cohomology, and the resulting map
$S^2\vee B\Spin^{[7]} \to B\Spin^c$ induces a
monomorphism in homology with image
\[
\F_2\{1,a_{1,0}^{(1)},a_{1,1}^{(1)},a_{3,0}^{(1)},a_{7,0}\}.
\]
The Thom spectrum over this space has a cell structure
of the form
\[
(S^0\cup_\eta e^2)\cup_\nu e^4\cup_\eta e^6\cup_2 e^7.
\]
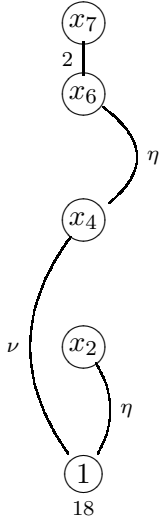
\[
\xymatrix@C=0.5cm@R=0.45cm{
 &  *+[o][F]{x_7}\ar@{-}[d]_{2} \\
 & *+[o][F]{x_6}\ar@/^20pt/@{-}[dd]^{\eta}  \\
 & \\
 & *+[o][F]{x_4}\ar@/_20pt/@{-}[dddd]_{\nu}   \\
 & \\
 & *+[o][F]{x_2}\ar@/^10pt/@{-}[dd]^{\eta}  \\
 &  \\
 & *+[o][F]{1}
}
\]

The skeletal inclusion factors through an infinite
loop map
\[
\xymatrix{
S^2\vee B\Spin^{[7]}\ar[dr]\ar[rr] && B\Spin^c \\
& \dlQ(S^2\vee B\Spin^{[7]})\ar[ur]_{\jc} & \\
}
\]
and we obtain an \Einfty Thom spectrum $M\jc$ over
$\dlQ(S^2\vee B\Spin^{[7]})$ whose homology is
\[
H_*(M\jc) =
\F_2[\dlQ^{I_2}x_2,\dlQ^{I_4}x_4,\dlQ^{I_6}x_6,\dlQ^{I_7}x_7
               : \text{\rm $I_r$ admissible, $\exc(I_r)>r$}].
\]
It is easy to see that there is a morphism of \Einfty
ring spectra
\[
\widetilde{\mathbb{P}}(S^0\cup_\nu e^4\cup_\eta e^6\cup_2 e^7)
\to k\U
\]
inducing an epimorphism on $H_*(-)$ under which
\[
x_2\mapsto \zeta_1^2,
\quad
x_4\mapsto \zeta_1^4,
\quad
x_6\mapsto \zeta_2^2,
\quad
x_7\mapsto \zeta_3.
\]
The $7$-skeleton of $M\jc$ has the form
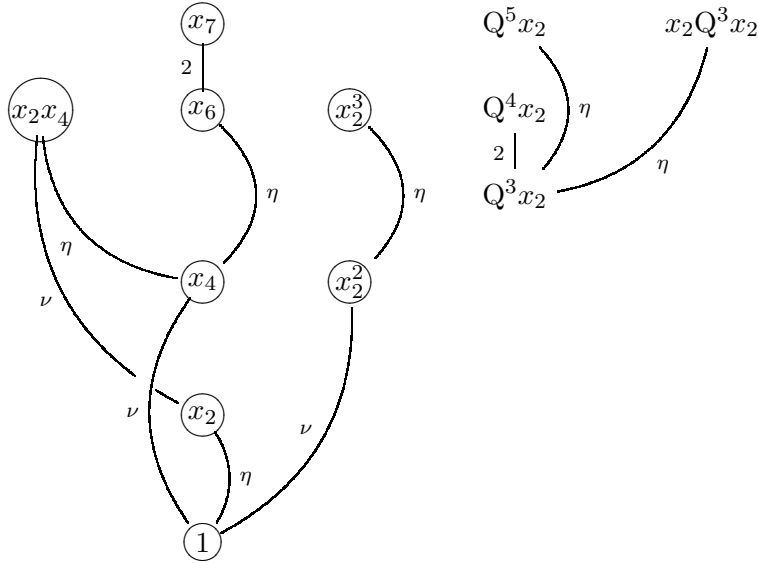
\[
\xymatrix@C=0.5cm@R=0.45cm{
&& *+[o][F]{x_7}\ar@{-}[d]_{2} && && \dlQ^5x_2\ar@/^20pt/@{-}[dd]^{\eta}
&& x_2\dlQ^3x_2\ar@/^20pt/@{-}[ddll]^{\eta} \\
*+[o][F]{x_2x_4^{\ph{2}}}\ar@/_20pt/@{-}[ddrr]_{\eta}\ar@{-}@/_25pt/[ddddrr]_{\nu}|(.84)\hole
&& *+[o][F]{x_6}\ar@/^20pt/@{-}[dd]^{\eta} && *+[o][F]{x_2^3}\ar@/^20pt/@{-}[dd]^{\eta}
&& \dlQ^4x_2\ar@{-}[d]_{2} && \\
&& && && \dlQ^3x_2 && \\
&& *+[o][F]{x_4}\ar@/_20pt/@{-}[dddd]_{\nu}
&& *+[o][F]{x_2^2}\ar@/^20pt/@{-}[lldddd]_{\nu} && && \\
&& && && && \\
&& *+[o][F]{x_2}\ar@/^10pt/@{-}[dd]^{\eta} && && && \\
&& && && &&  \\
&& *+[o][F]{1} && && &&
}
\]
since
$\pi_3(C_\eta)\iso\pi_3(S^0)/\eta\pi_1(S^0)=\pi_3(S^0)/4\pi_3(S^0)$
and the generators are detected by $\Sq^4$. It follows
that there is an element $\pi_4(M\jc)$ with Hurewicz
image $x_4+x_2^2$, and if $w\:S^4\to M\jc$ is a
representative, we can form the \Einfty cone $M\jc/\!/w$
as the pushout in the diagram
\[
\xymatrix{
\mathbb{P}S^4\ar[r]^w\ar[d]_{\tilde{w}}\ar@{}[dr]|(.3){\PO}
                                  & \mathbb{P}D^5\ar[d]  \\
{M\jc}\ar[r] & {M\jc/\!/w}
}
\]
taken in the category $\mathscr{C}_S$ of commutative
$S$-algebras. There is a K\"unneth spectral sequence
of the form
\[
\mathrm{E}^2_{s,t} =
\Tor^{H_*(\mathbb{P}S^4)}_{s,t}(\F_2,H_*(M\jc))
               \Lra H_{s+t}(M\jc/\!/w)
\]
where the $H_*(M\jc)$ is the $H_*(\mathbb{P}S^4)$-module
algebra
\[
H_*(\mathbb{P}S^4) =
\F_2[\dlQ^Iz_4 : \text{$I$ admissible, $\exc(I)>4$}]
\to H_*(M\jc);
\]
where
\[
\dlQ^Iz_4 \mapsto \dlQ^I(x_2^2) + \dlQ^Ix_4.
\]
Notice that the term $\dlQ^I(x_2^2)$ is either trivial
(if at least one term in $I$ is odd) or a square (if
all terms in $I$ are even), hence can be used as a
polynomial generator of $H_*(M\jc)$ in place of
$\dlQ^Ix_4$. It follows that $H_*(M\jc)$ is a free
$H_*(\mathbb{P}S^4)$-module, so the spectral sequence
is trivial with
\begin{align*}
\mathrm{E}^2_{*,*}
&= \Tor^{H_*(\mathbb{P}S^4)}_{0,*}(\F_2,H_*(M\jc)) \\
&= H_*(M\jc)/(\dlQ^I(x_2^2) + \dlQ^Ix_4
       : \text{$I$ admissible, $\exc(I)>4$}),
\end{align*}
therefore we have
\begin{equation}\label{eq:H*(Mjc//w)}
H_*(M\jc/\!/w) =
\F_2[\dlQ^{I_2}x_2,\dlQ^{I_6}x_6,\dlQ^{I_7}x_7
       : \text{\rm $I_r$ admissible, $\exc(I_r)>r$}].
\end{equation}
Here is the $7$-skeleton of $M\jc/\!/w$.
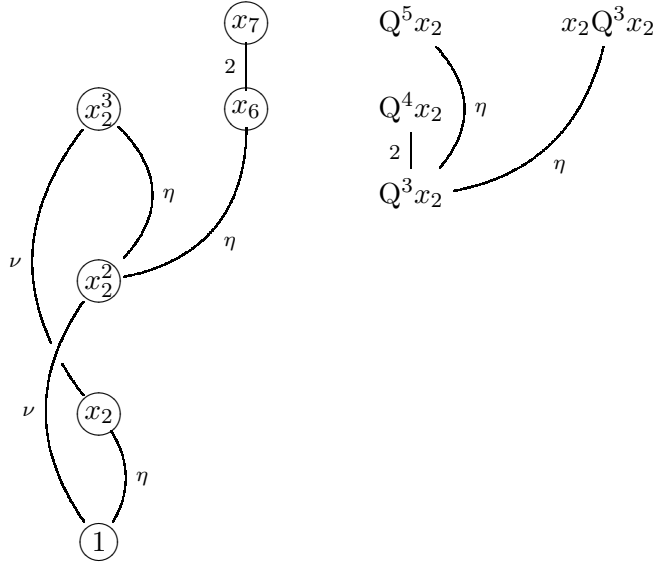
\[
\xymatrix@C=0.5cm@R=0.45cm{
&& *+[o][F]{x_7}\ar@{-}[d]_{2} && \dlQ^5x_2\ar@/^20pt/@{-}[dd]^{\eta}
&& x_2\dlQ^3x_2\ar@/^20pt/@{-}[ddll]^{\eta} \\
*+[o][F]{x_2^3}\ar@/^20pt/@{-}[dd]^{\eta}\ar@{-}@/_25pt/[dddd]_{\nu}|(.8)\hole
&& *+[o][F]{x_6}\ar@/^20pt/@{-}[ddll]^{\eta} && \dlQ^4x_2\ar@{-}[d]_{2} && \\
&& &&\dlQ^3x_2 && \\
*+[o][F]{x_2^2}\ar@/_20pt/@{-}[dddd]_{\nu} && && && \\
&& && && \\
*+[o][F]{x_2}\ar@/^10pt/@{-}[dd]^{\eta} && && && \\
&& && && \\
*+[o][F]{1} && && &&
}
\]

We define a sequence of elements $X_s$ in $H_*(M\jc/\!/w)$
by
\[
X_s =
\begin{dcases*}
x_2 & if $s=1$, \\
x_6 & if $s=2$, \\
x_7 & if $s=3$, \\
\dlQ^{(2^{s-1},\ldots,2^4,2^3)}x_7
   = \dlQ^{2^{s-1}}X_{s-1} & if $s\geq4$.
\end{dcases*}
\]
This is a regular sequence and the induced coaction
over the quotient Hopf algebra
\[
\mathcal{E}(1)_* =
\mathcal{A}_*/(\zeta_1^2,\zeta_2^2,\zeta_3,\ldots)
= \mathcal{A}_*/\!/\F_2[\zeta_1^2,\zeta_2^2,\zeta_3,\ldots]
= \Lambda(\zeta_1,\zeta_2)
\]
satisfies
\[
\psi'X_s =
\begin{dcases*}
1\otimes X_1 & if $s=1,2$, \\
1\otimes X_3 + \zeta_1\otimes X_2 + \zeta_2\otimes X_1^2 & if $s=3$, \\
1\otimes X_s + \zeta_1\otimes X_{s-1} + \zeta_2\otimes X_{s-2} & if $s\geq4$.
\end{dcases*}
\]
Therefore the ideal $I^{\mathrm{c}}=(X_s:s\geq1)\lhd H_*(M\jc/\!/w)$
is an $\mathcal{E}(1)_*$-invariant regular ideal.

Recall that
\[
\mathcal{A}_*\square_{\mathcal{E}(1)_*}\F_2
= \F_2[\zeta_1^2,\zeta_2^2,\zeta_3,\ldots]
\iso H_*(k\U).
\]
We have proved the following analogues of earlier
results.
\begin{prop}\label{prop:H*Mjc-extended}
There is an isomorphism of $\mathcal{A}_*$-comodule
algebras
\[
H_*(M\jc/\!/w) \xrightarrow{\;\iso\;}
\mathcal{A}_*\square_{\mathcal{E}(1)_*} H_*(M\jc/\!/w)/I^{\mathrm{c}}.
\]
\end{prop}

\begin{prop}\label{prop:H*Mjc-splitting}
There is a splitting of $\mathcal{A}_*$-comodule
algebras
\[
\xymatrix{
\mathcal{A}_*\square_{\mathcal{E}(1)_*}\F_2\ar[rr]^{\iso}\ar[dr]
&& \mathcal{A}_*\square_{\mathcal{E}(1)_*}\F_2 \\
& H_*(M\jc/\!/w)\ar[ur] &
}
\]
where
$H_*(M\jc/\!/w)\to H_*(k\U)=\mathcal{A}_*\square_{\mathcal{E}(1)_*}\F_2$
is induced by a factorisation $M\jc\to M\jc/\!/w\to k\U$
of the \Einfty orientation.
\end{prop}

Of course, in principle use of the well-known lightning
flash technology of~\cites{JFA:Chicago,JFA&SBP:BSO} should
lead to a description of $H_*(M\jc/\!/w)/I^{\mathrm{c}}$
as an $\mathcal{E}(1)_*$-comodule. For example, there
are many infinite lightning flashes such as the following

\[
\xymatrix{
&& \dlQ^4x_2 \ar[dl]_{\q^0_*} && \dlQ^6x_2 \ar@{->>}[dlll]_{\q^1_*} \ar[dl]
&& \dlQ^8x_2 \ar[dl]\ar@{->>}[dlll] && \ar@{.>>}[dlll]&&&&&&&& \\
& \dlQ^3x_2 && \dlQ^5x_2 && \dlQ^7x_2 &&&&&&&&&&&
}
\]
as well as parallelograms such as
\[
\xymatrix{
&&&&&\dlQ^8\dlQ^4x_2\ar[ddl]\ar@{->>}[dll]& \\
&&&\dlQ^6\dlQ^3x_2\ar[ddll]_{\q^0_*}&&& \\
&&&&\dlQ^7\dlQ^4x_2\ar@{->>}[dlll]_{\q^1_*} &&  \\
&\dlQ^5\dlQ^3x_2 &&&&&
}
\]
which can be determined by using~\cite{Nishida}*{proposition~7.3}.

\subsection{An example related to the Brown-Peterson
                     spectrum}\label{subsec:Examp-BP}
{}From~\cite{BP-Einfinity}*{section~4} we recall
the $2$-local \Einfty ring spectrum $R_\infty$
for which there is a map of commutative ring
spectra $R_\infty\to BP$ inducing a rational
equivalence, an epimorphism
$\pi_*(R_\infty)\to\pi_*(BP)$, and $H_*(R_\infty)$
contains a regular sequence
$z_s\in H_{2^{2+1}-2}(R_\infty)$ mapping to the
generators $t_s\in H_{2^{2+1}-2}(BP)$ which in
turn map to
$\zeta_s^2\in H_{2^{2+1}-2}(H)=\mathcal{A}_{2^{2+1}-2}$
under the induced ring homomorphisms
\[
H_*(R_\infty)\to H_*(BP)\to H_*(H)=\mathcal{A}_*.
\]
We note that both of these homomorphisms are
compatible with the Dyer-Lashof operations,
even though~$BP$ is not known to be an \Einfty
ring spectrum. These elements $z_s$ have the
following coactions:
\[
\psi(z_r) =
   1\otimes z_r + \zeta_1^2\otimes z_{r-1}^2
    + \zeta_2^2\otimes z_{r-2}^{4}
    + \cdots + \zeta_{r-1}^2\otimes z_{1}^{2^{r-1}}
    + \zeta_r^2\otimes 1,
\]
and generate an ideal $I_\infty \lhd H_*(R_\infty)$.

Let
\[
\mathcal{E}_* = \mathcal{A}_*/(\zeta_i^2 : i\geq1),
\]
the exterior quotient Hopf algebra. Although
$\mathcal{E}_*$ is not finite dimensional,
it is still true that $\mathcal{A}_*$ is an
extended right $\mathcal{E}_*$-comodule,
\[
\mathcal{A}_* \iso
(\mathcal{A}_*\square_{\mathcal{E}_*}\F_2)\otimes\mathcal{E}_*.
\]
Under the induced $\mathcal{E}_*$-coaction
on $H_*(R_\infty)$, $I_\infty$ is an
$\mathcal{E}_*$-comodule ideal, therefore
$H_*(R_\infty)/I_\infty$ is an $\mathcal{E}_*$-comodule
algebra.
\begin{prop}\label{prop:Rinfty}
There is an isomorphism of commutative $\mathcal{A}_*$-comodule
algebras
\[
H_*(R_\infty) \xrightarrow{\;\iso\;}
\mathcal{A}_*\square_{\mathcal{E}_*}H_*(R_\infty)/I_\infty,
\]
and a splitting of $\mathcal{A}_*$-comodule algebras
\[
\xymatrix{
\mathcal{A}_*\square_{\mathcal{E}_*}\F_2\ar[rr]^{\iso}\ar[dr]
&& \mathcal{A}_*\square_{\mathcal{E}_*}\F_2 \\
& H_*(R_\infty)\ar[ur] &
}
\]
where $\mathcal{A}_*\square_{\mathcal{E}_*}\F_2\iso H_*(BP)$
and the right hand homomorphism is induced from the
morphism of commutative ring spectra $R_\infty\to BP$.
\end{prop}

This result supports the view that $R_\infty$ admits
a map $BP\to R_\infty$ extending the unit $S^0\to R_\infty$
and then the composition
\[
BP \to R_\infty\to BP
\]
would necessarily be a weak equivalence since~$BP$
is minimal atomic in the sense of~\cite{AJB&JPM}.

%

\section{Speculation and conjectures}\label{sec:Speculation}

Our algebraic splittings of $H_*(Mj_r)$ are consistent
with spectrum-level splittings. Indeed, in the case
of $r=1$, a result of Mark Steinberger~\cite{LNM1176}
already shows that $Mj_1$ splits as a wedge of suspensions
of $H\Z$ and $H\Z/2^s$ for $s\geq1$, all of which are
$H\Z$-module spectra. In fact a direct argument is also
possible.

Using Lemma~\ref{lem:Splitting-Alg-gen}, it is easy
to see that if a spectrum $X$ is a module spectrum
over one of $H\Z$, $k\O$ or $\tmf$ then its homology
is a retract of the extended comodule
$\mathcal{A}_*\square_{\mathcal{A}(r)_*}H_*(X)$ for
the relevant value of~$r$; a similar observation holds
for a module spectrum over $k\U$ and
$\mathcal{A}_*\square_{\mathcal{E}(1)_*}H_*(X)$. Thus
our algebraic results provide evidence for the following
conjectural splittings.
\begin{conj}\label{conj:Splittings}
As a spectrum, $Mj_2$ is a wedge of $k\O$-module spectra,
$Mj_3$ is a wedge of $\tmf$-module spectra and $M\jc$
is a wedge of $k\U$-module spectra.
\end{conj}

Here the phrase `module spectrum' can be interpreted either
purely homotopically, or strictly in the sense of~\cite{EKMM}.
In each case, it is enough to produce any map $E\to Mj$
extending the unit (up to homotopy), for then the \Einfty
structure on $Mj$ gives rise to a homotopy commutative
diagram of the following form.
\[
\xymatrix{
&&&& \\
S^0\wedge Mj\ar[r]\ar[dr]\ar@/^30pt/[rrrr]^{\sim} & Mj\wedge Mj\ar[r]
& \tilde{\mathbb{P}}Mj\wedge Mj\ar[r] & Mj\wedge Mj\ar[r] & Mj \\
&E\wedge Mj\ar[u]\ar[r] &\tilde{\mathbb{P}}E\wedge Mj\ar[u]&& \\
}
\]

Related to this conjecture, and indeed implied by it, is
the following where we know that analogues hold for the
cases $Mj_1$, $Mj_2$, $M\jc$, i.e., the natural homomorphisms
\[
\pi_*(Mj_1)\to \pi_*(H\Z),
\quad
\pi_*(Mj_2)\to \pi_*(k\O),
\quad
\pi_*(M\jc)\to \pi_*(k\U)
\]
are epimorphisms.
One approach to verifying these is by using the Adams spectral
sequence: in each of the first two cases the lowest degree
element in the $\mathrm{E}_2$-term not associated with the
$\mathcal{A}_*\square_{\mathcal{A}(r-1)_*}\F_2$ summand is
one of the elements $\dlQ^3x_2$ or $\dlQ^5x_4$ and this is
too far along to give elements supporting anomalous differentials
on this summand, and the multiplicative structure completes
the argument. Here is a small portion of the Adams spectral
sequence for $Mj_2$ to illustrate this, with $\dlQ^5x_4$
at position $(9,0)$ and most of the diagram being part
of the $\mathrm{E}_2$-term for $k\O$. Since
\[
\psi\dlQ^6x_4 = \zeta_1\otimes\dlQ^5x_4 + 1\otimes\dlQ^6x_4,
\]
this element  $\dlQ^5x_4$ does not produce an~$h_0$ tower;
in fact the $\mathcal{A}(1)_*$-subcomodule
\[
\F_2\{\dlQ^5x_4,\dlQ^6x_4\}\subseteq H_*(Mj_2)/I_2
\]
gives rise to a copy of the Adams $\mathrm{E}_2$-term for
$k\O\wedge(S^0\cup_2 e^1)$ carried on~$\dlQ^5x_4$.
\bigskip
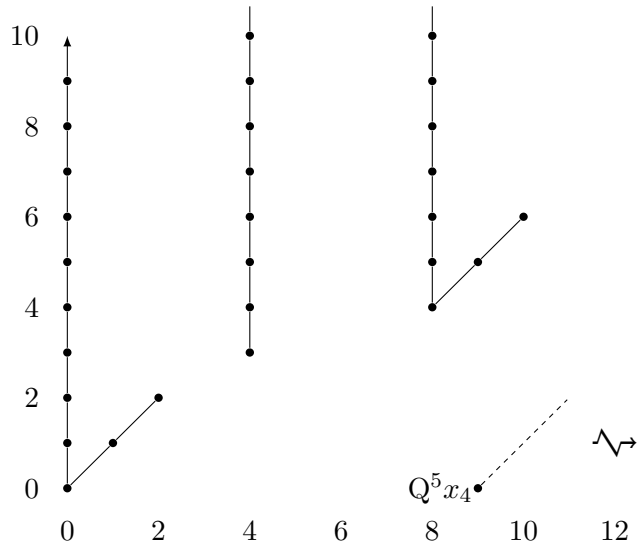
\begin{center}
\begin{sseq}[arrows=latex, grid=none, entrysize=6mm]
{0...12}{0...10}
\ssinfbullstring 0 1 {10}
\ssmoveto 0 0
\ssline 1 1
\ssbullstring 1 1 {2}
\ssmoveto 4 3
\ssbullstring 0 1 {10}
\ssmoveto 8 4
\ssbullstring 0 1 {10}
\ssmoveto 8 4
\ssline 1 1
\ssbullstring 1 1 {2}
\ssmoveto 9 0
\ssdropbull
\ssdroplabel[L]{{\dlQ^5x_4}}
\ssmove 2 2
\ssdrop{}
\ssdashedcurve{0}
\ssmoveto {12} 1 \ssdrop{$\manerrarrow$}
\end{sseq}
\end{center}
In the third case, the first element not in the $k\U$
summand is $\dlQ^3x_2$ and a similar argument applies.

\begin{conj}\label{conj:HtpyEpi}
The \Einfty orientation $Mj_3\to\tmf$ induces a ring
epimorphism $\pi_*(Mj_3)\to\pi_*(\tmf)$.
\end{conj}

This is easily seen to be true up to degree~$16$, and
it also holds rationally. To go further seems to require
detailed examination of the Adams spectral sequences
for $\pi_*(Mj_3)$ and $\pi_*(\tmf)$, and to date we
have checked it up to degree~$26$. Of course this
conjecture is implied by the above splitting conjecture.

To understand how the splitting question might be resolved,
let us examine the settled case of $Mj_1$. This provides
a universal example for the general splitting result of
Steinberger~\cite{LNM1176}*{theorem~III.4.2}, and the
general case is implied by that of $Mj_1$. Since
\[
H_*(Mj_1)\iso\mathcal{A}_*\square_{\mathcal{A}(0)_*}H_*(Mj_1)/I_1,
\]
we have
\[
\Ext_{\mathcal{A}_*}^{*,*}(\F_2,H_*(Mj_1)) \iso
\Ext_{\mathcal{A}(0)_*}^{*,*}(\F_2,H_*(Mj_1)/I_1).
\]
Following the strategy of Steinberger's proof for the general
case, we consider the $\mathcal{A}(0)_*$-comodule structure
of $H_*(Mj_1)/I_1$, or equivalently its $\mathcal{A}(0)^*$-module
structure. Of course here there is only one copy of $H\Z$, the
remaining summands are suspensions of $H\Z/2^r$ for various~$r$.

The Bockstein spectral sequence for $H_*(Mj_1;\Z_{(2)})$
can be determined from this using formulae for higher
Bocksteins of~\cite{JPM:SteenrodOps}*{proposition~6.8},
which we learnt about from Rolf Hoyer and Peter May.

Let $E$ be a connective finite type $2$-local \Einfty
ring spectrum and let $x\in H_{2m}(E)$ where $m\in\Z$.
Writing $\beta_k$ for the $k$-th higher Bockstein
operation, and assuming that $\beta_{k-1}x$ is defined,
we have
\begin{equation}\label{eq:HigherBockstein}
\beta_k(x^2) =
\begin{dcases*}
x\beta_1x + \dlQ^{2m}(\beta_1x)  & if $k=2$, \\
x\beta_{k-1}x & if $k>2$.
\end{dcases*}
\end{equation}
These formulae determine higher differentials in the
Bockstein spectral sequence for $H_*(E;\Z_{(2)})$.
The first differential $\beta_1=\Sq^1_*$ is given on
polynomial generators by
\begin{align}
\beta_1\dlQ^Ix_2 &=
\begin{dcases*}
\dlQ^{(i_1-1,i_2,\ldots,i_k)}x_2
& if $k>0$ and $I=(i_1,i_2,\ldots,i_k)$ with $i_1$ even, \\
0 & otherwise,
\end{dcases*}
                           \label{eqn:beta1-x2} \\
\beta_1\dlQ^Ix_3 &=
\begin{dcases*}
x_2 & if $I=()$ is the empty sequence, \\
\dlQ^{(i_1-1,i_2,\ldots,i_k)}x_3
& if $k>0$ and $I=(i_1,i_2,\ldots,i_k)$ with $i_1$ even, \\
0 & otherwise.
\end{dcases*}
                           \label{eqn:beta1-x3}
\end{align}
In each of the cases with $i_1$ even, $\dlQ^{(i_1-1,i_2,\ldots,i_k)}x_s$
is a polynomial generator except when $i_1=i_2+\cdots+i_k+s+1$
and then
\[
\beta_1\dlQ^Ix_s = (\dlQ^{(i_2,\ldots,i_k)}x_s)^2.
\]
As a dga with respect to $\beta_1$, $H_*(Mj_1)$ is a
tensor product of acyclic subcomplexes of the form
$\F_2[\beta_1\dlQ^Ix_s,\dlQ^Ix_s]$ where $s=2,3$ and
$I=(i_1,\ldots,i_k)\neq()$ with $i_1$ even, together
with $\F_2[x_2,x_3]$ and the polynomial ring generated
by the squares not already accounted for. In particular,
the $\mathrm{E}^2$-term of the Bockstein spectral sequence
agrees with the $\beta_1$-homology of $H_*(Mj_1)/I_1$.
The higher Bocksteins now follow from the above formulae
\eqref{eqn:beta1-x2} and \eqref{eqn:beta1-x3}.

This approach might be generalised to the cases of $Mj_2$,
$Mj_3$ and $M\jc$ by studying suitable Bockstein spectral
sequences for $k\O_*(Mj_2)$, $\tmf_*(Mj_3)$ and $k\U_*(M\jc)$.
We remark that the \Einfty ring spectra $H\Z\wedge Mj_1$,
$k\O\wedge Mj_2$ and $\tmf\wedge Mj_3$ can be identified
in different guises using the Thom diagonals associated
with the \Einfty orientations $Mj_1\to H\Z$, $MJ_2\to k\O$
and $Mj_3\to \tmf$, giving weak equivalences of \Einfty
ring spectra
\begin{align*}
H\Z\wedge Mj_1&\xrightarrow{\;\sim\;}H\Z\wedge\Sigma^\infty_+\dlQ(B\SO^{[3]}), \\
k\O\wedge Mj_2&\xrightarrow{\;\sim\;}k\O\wedge\Sigma^\infty_+\dlQ(B\Spin^{[7]}), \\
\tmf\wedge Mj_3&\xrightarrow{\;\sim\;}\tmf\wedge\Sigma^\infty_+\dlQ(B\String^{[15]}),
\end{align*}
and there are isomorphisms of $\mathcal{A}_*$-comodule algebras
\begin{align*}
H_*(H\Z\wedge Mj_1) &\xrightarrow{\iso}
\mathcal{A}_*\square_{\mathcal{A}(0)_*}H_*(\dlQ(B\SO^{[3]})), \\
H_*(k\O\wedge Mj_2) &\xrightarrow{\iso}
\mathcal{A}_*\square_{\mathcal{A}(1)_*}H_*(\dlQ(B\Spin^{[7]})), \\
H_*(\tmf\wedge Mj_3) &\xrightarrow{\iso}
\mathcal{A}_*\square_{\mathcal{A}(2)_*}H_*(\dlQ(B\String^{[15]})).
\end{align*}

\bigskip
The Referee has raised the question of whether
the approach of Subsection~\ref{subsec:Examp-kU}
can be used to produce an \Einfty Thom spectrum
related to $\tmf_1(3)$ as $M\jc$ is related to
$k\U$. We recall from~\cite{TL&NN} that there
is a commutative diagram of $2$-local \Einfty
ring spectra
\[
\xymatrix{
\tmf\ar[r]\ar[d] & k\O\ar[d] \\
\tmf_1(3)\ar[r] & k\U
}
\]
On applying $H_*(-)$, this induces the following
diagram of $\mathcal{A}_*$-comodule subalgebras
of~$\mathcal{A}_*$.
\[
\xymatrix{
\mathcal{A}_*\square_{\mathcal{A}(2)_*}\F_2\ar@{=}[r]
& \F_2[\zeta_1^8,\zeta_2^4,\zeta_3^2,\zeta_4,\ldots]
\ar@{^{(}->}[d]\ar@{^{(}->}[r]
&\F_2[\zeta_1^4,\zeta_2^2,\zeta_3,\ldots]\ar@{^{(}->}[d]
& \mathcal{A}_*\square_{\mathcal{A}(1)_*}\F_2\ar@{=}[l]  \\
\mathcal{A}_*\square_{\mathcal{E}(2)_*}\F_2\ar@{=}[r]
&\F_2[\zeta_1^2,\zeta_2^2,\zeta_3^2,\zeta_4,\ldots]\ar@{^{(}->}[r]
&\F_2[\zeta_1^2,\zeta_2^2,\zeta_3,\ldots]
&\mathcal{A}_*\square_{\mathcal{E}(1)_*}\F_2\ar@{=}[l]
}
\]
We propose using the space
\[
S^2\vee B\Spin^{[6]}\vee B\O\langle8\rangle^{[15]}
\]
which admits a map to $B\Spinc$ that restricts
to a map inducing an epimorphism in cohomology
on each wedge summand. Extending this to an
infinite loop map
\[
j\:\dlQ(S^2\vee B\Spin^{[6]}\vee B\O\langle8\rangle^{[15]})
 \to B\Spinc \to B\SO
\]
we obtain an \Einfty Thom spectrum~$Mj$.
\begin{conj}\label{conj:tmf1(3)}
There is an \Einfty morphism~$Mj\to\tmf_1(3)$
which factors through an \Einfty $3$-cell
complex $Mj/\!/w_4,w_8,w_{12}$ with \Einfty
cells of dimensions~$5$, $9$ and~$13$ attached
by maps $w_4,w_8,w_{12}$. Moreover the morphism
$Mj/\!/w_4,w_8,w_{12}\to\tmf_1(3)$ induces
an epimorphism on $H_*(-)$ which is an
isomorphism up to degree~$15$.
\end{conj}

We have not yet checked all the details,
but it seems plausible that the approach
used for $M\jc$ offers a route to doing
this. Of course we might then expect a
splitting of $Mj/\!/w_4,w_8,w_{12}$ into
$\tmf_1(3)$-module spectra, or at least
that the map $Mj/\!/w_4,w_8,w_{12}\to\tmf_1(3)$
induces an epimorphism on $\pi_*(-)$.

\bigskip
\appendix
\section{On the homology of connective covers
of $B\O$}\label{sec:HomConnCov}

We review the structure of the homology Hopf algebras
$H_*(B\O\langle n\rangle;\F_2)$ for $n=1,2,4,8$. The
dual cohomology rings were originally determined by
Stong, but later a body of literature due to Bahri,
Kochman, Pengelley as well as the present author
evolved describing these homology rings. We will use
the Husemoller-Witt decompositions of~\cite{H-W&StAlg}
to give explicit algebra generators; the actions of
Steenrod and Dyer-Lashof operations on these can be
determined using work of Kochman and
Lance~\cites{SOK:DLops,TL:St&DLops}.

We recall that there are polynomial generators
$a_{k,s}\in H_{2^sk}(B\O)$ ($k$ odd, $s\geq0$)
such that
\[
\mathrm{B}[k]_* = \F_2[a_{k,s}:s\geq0]
     \subseteq H_*(B\O)
\]
is a polynomial sub-Hopf algebra and there
is a decomposition of Hopf algebras
\[
H_*(B\O) =
\bigotimes_{\text{$k$ odd}}\mathrm{B}[k]_*.
\]

For each odd~$k$ there is an isomorphism
of Hopf algebras
\[
\mathrm{B}[k]^*/(a_{k,0})
      \iso\Hom(\mathrm{B}^{(1)}[k]_*,\F_2),
\]

Here the dual Hopf algebra
$\mathrm{B}[k]^*=\Hom(\mathrm{B}[k]_*,\F_2)$
is isomorphic to $\mathrm{B}[k]_*$, i.e.,
these are self dual Hopf algebras. There is
also a decomposition of Hopf algebras
\[
H^*(B\O) = \bigotimes_{\text{$k$ odd}}B[k]^*.
\]

For each $h\geq1$, there is a monomorphism of
Hopf algebras which multiplies degrees by~$2^h$,
\[
\mathrm{B}[k]_*\to\mathrm{B}[k]_*;
\quad
x\mapsto x^{(h)}= x^{2^h},
\]
whose image is denoted by $\mathrm{B}^{(h)}[k]_*$.
Notice that the primitives in $\mathrm{B}^{(h)}[k]_*$
are the powers
\[
(a^{(h)}_{k,0})^{2^s} = (a_{k,0})^{2^{s+h}}
                           \quad (s\geq0).
\]
Dually there is an epimorphism of Hopf algebras
\[
\mathrm{B}[k]^*\to\mathrm{B}[k]^*;
\quad
a_{k,s}\mapsto
\begin{dcases*}
a_{k,s-h} & if $s\geq h$, \\
\ph{abc}
0 & if $s<h$,
\end{dcases*}
\]
and this induces an isomorphism of Hopf algebras
\[
\mathrm{B}[k]^*/(a_{k,0},a_{k,1},\ldots,a_{k,h-1})
  \iso \mathrm{B}[k]^*
\]
which divides degrees by $2^h$. The dual Hopf algebra
of $\mathrm{B}^{(h)}[k]_*$ is
\[
\mathrm{B}^{(h)}[k]^* \iso
\mathrm{B}[k]^*/(a_{k,0},a_{k,1},\ldots,a_{k,h-1}).
\]

Let $\alpha=\alpha_2$ denote the dyadic number
function which counts the number of non-zero
coefficients in the binary expansion of a natural
number.

\begin{thm}\label{thm:H*BO<n>}
The natural infinite loop maps
$B\O\langle n\rangle\to B\O\langle1\rangle=B\O$
$(n=2,4,8)$ induce monomorphisms of Hopf algebras
$H_*(B\O\langle n\rangle)\to H_*(B\O)$ whose
images are the following sub-Hopf algebras of
$H_*(B\O)$:
\begin{align*}
&\mathrm{B}^{(1)}[1]_*\otimes
      \bigotimes_{\text{\rm odd $k>1$}}\mathrm{B}[k]_*
&\quad &\text{\rm if $n=2$},  \\
&\mathrm{B}^{(2)}[1]_*\otimes
\bigotimes_{\substack{\text{\rm $k$ odd} \\
\alpha(k)=2}}\mathrm{B}^{(1)}[k]_*\otimes
\bigotimes_{\substack{\text{\rm $k$ odd} \\
\alpha(k)>2}}\mathrm{B}[k]_*
&\quad &\text{\rm if $n=4$},  \\
&\mathrm{B}^{(3)}[1]_*\otimes
\bigotimes_{\substack{\text{\rm $k$ odd} \\
\alpha(k)=2}}\mathrm{B}^{(2)}[k]_*
\otimes
\bigotimes_{\substack{\text{\rm $k$ odd} \\
\alpha(k)=3}}\mathrm{B}^{(1)}[k]_*
\otimes\bigotimes_{\substack{\text{\rm $k$ odd} \\
\alpha(k)>3}}\mathrm{B}[k]_*
&\quad &\text{\rm if $n=8$}.
\end{align*}
\end{thm}

By dualising and using the above observations
we obtain Hopf algebra decompositions of the
cohomology of these spaces. For example,
\begin{align*}
H^*(B\SO) = H^*(B\O\langle2\rangle)&=
\mathrm{B}^{(1)}[1]^*
\otimes \bigotimes_{\text{\rm odd $k>1$}}\mathrm{B}[k]^* \\
&= \mathrm{B}[1]^*/(a_{1,0})
\otimes \bigotimes_{\text{\rm odd $k>1$}}\mathrm{B}[k]^*.
\end{align*}

We may identify $H_*(M\O\langle n\rangle)$ with
$H_*(B\O\langle n\rangle)$ using the Thom isomorphism
which is an isomorphism of algebras over the Dyer-Lashof
algebra but not over the Steenrod algebra. To avoid
excessive notation we will often treat the Thom isomorphism
as an equality and write $a_{k,s}^{(r)}$ for each of
the corresponding elements.

The generators $a_{2^s-1,0}$ are particularly interesting.
In $H_*(B\O)$, $a_{2^s-1,0}$ is primitive, and in $H_*(M\O)$
there is a simple formula for the $\mathcal{A}_*$-coaction:
\begin{equation}\label{eq:coaction-a}
\psi(a_{2^s-1,0}) =
1\otimes a_{2^s-1,0} + \zeta_1\otimes a_{2^{s-1}-1,0}^2
+ \zeta_2\otimes a_{2^{s-2}-1,0}^4 + \cdots
+ \zeta_{s-1}\otimes a_{1,0}^{2^{s-1}} + \zeta_s\otimes1.
\end{equation}
The natural orientation $M\O\to H\F_2$ induces an
algebra homomorphism over both of the Dyer-Lashof
and Steenrod algebras under which
\begin{equation}\label{eq:ak0->zeta}
a_{2^s-1,0}\mapsto\zeta_s.
\end{equation}

For completeness, we also describe the homology
of $B\Spinc$ in similar algebraic form to that
of Theorem~\ref {thm:H*BO<n>} since we are not
aware of this being documented anywhere else;
note that~\cite{RES:Book}*{page~293} contains
an apparently incorrect statement on the mod~$2$
cohomology, while~\cite{Harada&Kono} describes
the cohomology of $B\Spinc(n)$.
\begin{thm}\label{thm:Spinc}
The natural infinite loop map $B\Spinc\to B\O$
induces a monomorphism of Hopf algebras
$H_*(B\Spinc)\to H_*(B\O)$ with image
\[
\bigotimes_{\substack{\text{\rm $k$ odd}\\ \alpha(k)\leq2}}\mathrm{B}^{(1)}[k]_*
\otimes
\bigotimes_{\substack{\text{\rm $k$ odd}\\ \alpha(k)>2}}\mathrm{B}[k]_*.
\]
\end{thm}
\begin{proof}[Sketch of proof]
The cohomology ring $H^*(B\Spinc)$ can be calculated
using the Serre spectral sequence
\[
\mathrm{E}_2^{r,s} = H^r(B\SO;H^s(K(\Z,2)))
                     \Lra H^{r+s}(B\Spinc)
\]
for the fibration sequence
\[
K(\Z,2)\to B\Spinc \to B\SO.
\]
Then
\[
\mathrm{E}_2^{*,*} = \F_2[w_k: k\geq2]\otimes\F_2[x],
\]
where $w_k\in H^k(B\SO)$ is the image of the
$k$-th Stiefel-Whitney class, while $x\in H^2(K(\Z,2))$
and $x^{2^t}\in H^{2^{t+1}}(K(\Z,2))$ transgresses
to
\[
d_{2^{t+1}+1}(x^{2^t}) = w_{2^{t+1}+1} \pmod{\text{decomposables}}.
\]
As $d_{2^{t+1}+1}(x^{2^t})$ has to be a primitive,
it must agree with the element $a_{2^{t+1}+1,0}$.
It follows that the natural map $B\Spinc\to B\O$
induces an epimorphism $H^*(B\O)\to H^*(B\Spinc)$,
while dually $H_*(B\Spinc)\to H_*(B\O)$ is a
monomorphism. Also $H^*(B\Spinc)$ is polynomial
with one generator in each degree~$k$ where either
$\alpha(k)>2$ or $k$ is even with $\alpha(k)\leq2$.
Indeed there is an isomorphism of Hopf algebras
\[
H^*(B\Spinc) \iso
\bigotimes_{\substack{\text{\rm $k$ odd}\\ \alpha(k)\leq2}}\mathrm{B}[k]^*/(a_{k,0})
\otimes
\bigotimes_{\substack{\text{\rm $k$ odd}\\ \alpha(k)>2}}\mathrm{B}[k]^*.
\]
The claimed description of the homology $H_*(B\Spinc)$
follows.
\end{proof}
\begin{rem}\label{rem:Spinc}
The natural map $B\Spin\to B\Spinc$ induces
a homomorphism in homology whose image
contains $(a_{1,0}^{(1)})^2$, $a_{3,0}^{(1)}$
and $a_{7,0}$.
\end{rem}

\section{Dyer-Lashof operations and Steenrod coactions}
\label{sec:DL&StAct}

For the convenience of the reader, we summarise some
results from~\cite{Nishida} which are based on work
of Kochman and Steinberger~\cites{SOK:DLops,LNM1176}.

The mod~$2$ Steenrod algebra $\mathcal{A}_*$ is the
homology of the mod~$2$ Eilenberg-Mac~Lane spectrum
$H=H\F_2$ which is an \Einfty ring spectrum and so
$\mathcal{A}_*$ supports an action of the Dyer-Lashof
operations. However, when dealing with the left
$\mathcal{A}_*$-coaction on the homology of an \Einfty
ring spectrum it is often convenient to consider a
twisted version formed using the antipode $\chi$
and given by
\[
\tdlQ^s = \chi\dlQ^s\chi.
\]
Based on Steinberger's determination of the usual
action~\cite{LNM1176}, by~\cite{Nishida}*{lemma~4.4}
we have the following equivalent formulae for all
$s\geq1$:
\begin{subequations}\label{eq:DL-A*}
\begin{align}
\dlQ^{2^s}\xi_s &= \xi_{s+1} + \xi_1\xi_s^2,
                   \label{eq:DL-xi} \\
\tdlQ^{2^s}\zeta_s &= \zeta_{s+1} + \zeta_1\zeta_s^2.
\label{eq:tDL-zeta}
\end{align}
\end{subequations}

The spectra $H\Z$, $k\O$ and $\tmf$ are all \Einfty  ring
spectra and there are \Einfty  morphisms $H\Z\to H\F_2$,
$k\O\to H\F_2$ and $\tmf\to H\F_2$ inducing monomorphisms on
$H_*(-)$ identifying their homology with the subalgebras
\[
\F_2[\zeta_1^8,\zeta_2^4,\zeta_3^2,\zeta_4,\zeta_5,\ldots]
\subseteq
\F_2[\zeta_1^4,\zeta_2^2,\zeta_3,\zeta_4,\ldots]
\subseteq
\F_2[\zeta_1^2,\zeta_2,\zeta_3,\ldots]
\subseteq\mathcal{A}_*.
\]
It follows that each of these subalgebras is closed under
the Dyer-Lashof operations. More generally by work of
Stong~\cite{RES:ConnCov}, each of the \Einfty  morphisms
$M\O\langle2^d\rangle\to H\F_2$ induces a ring homomorphism
whose image is
$\F_2[\zeta_1^{2^d},\zeta_2^{2^{d-1}},\ldots,
             \zeta_{d}^2,\zeta_{d+1},\zeta_{d+2},\ldots]$
and this must be closed under the Dyer-Lashof operations.

We will give a purely algebraic generalisation of these
observations.

For $n\geq0$, let
\[
\mathcal{I}(n) = (\zeta_1^{2^{n+1}},\zeta_2^{2^{n}},\zeta_3^{2^{n-1}},
  \ldots,\zeta_n^{4},\zeta_{n+1}^2,\zeta_{n+2},\zeta_{n+3},\ldots)
  \lhd \mathcal{A}_*.
\]
This is a Hopf ideal and
$\mathcal{A}(n)_* = \mathcal{A}_*/\mathcal{I}(n)$ is a
well-known finite quotient Hopf algebra. We also set
\[
\mathcal{I}(n)^{[d]}
= \{ \alpha^{2^d} : \alpha\in\mathcal{I}(n) \}
\lhd \mathcal{A}_*,
\]
and observe that
\begin{equation}\label{eq:I(n)[1]-I(n+1)}
\mathcal{I}(n)^{[d+1]} \subseteq \mathcal{I}(n+1)^{[d]}
            \subseteq \mathcal{I}(n+d).
\end{equation}
\begin{lem}\label{lem:Qkzetas}
Let $s\geq1$. If $k\in\N$, then
$\dlQ^k \zeta_s \in\mathcal{I}(s-1)$; more generally,
for $r\geq0$, $\dlQ^k(\zeta_s^{2^r}) \in\mathcal{I}(s+r-1)$.
\end{lem}
\begin{proof}
We make use of the results of \cite{Nishida}*{section~5}.

The proof is by induction on~$s$. When $s=1$, for $k\geq1$,
write $k=2m$ or $k=2m+1$. Then
\[
\dlQ^{2m}\zeta_1
= \mathrm{N}_{2m+1}(\xi)
= \xi_1 \mathrm{N}_{m}(\xi)^2 + \xi_2 \mathrm{N}_{m-1}(\xi)^2
       + \xi_3\mathrm{N}_{m-3}(\xi)^2 + \cdots
 \in \mathcal{I}(0),
\]
and
\begin{align*}
\dlQ^{2m+1}\zeta_1
&= \mathrm{N}_{2m+2}(\xi) = \mathrm{N}_{m+1}(\xi)^2 \\
&= \xi_1^2 \mathrm{N}_{m}(\xi)^4 + \xi_2^2 \mathrm{N}_{m-2}(\xi)^4
       + \xi_3^2 \mathrm{N}_{m-6}(\xi)^2 + \cdots
 \in \mathcal{I}(0).
\end{align*}

Now suppose that the result holds for all~$s<n$.
Recall that for $k\geq2^n-1$, $\dlQ^k\zeta_n=0$
unless $k\equiv 0\bmod{2^n}$ or $k\equiv2^n-1\bmod{2^n}$
when
\begin{align*}
\dlQ^{2^nm}\zeta_n
&= \mathrm{N}_{2^nm + 2^n-1}(\xi) \\
&= \xi_1\mathrm{N}_{2^{n-1}m + 2^{n-1}-1}(\xi)^2
 + \xi_2\mathrm{N}_{2^{n-2}m + 2^{n-2}-1}(\xi)^4 \\
& \ph{aaaaaaaaaaaaaaaaaaaaaaaaaaaaaaaa}
+ \xi_3\mathrm{N}_{2^{n-3}m + 2^{n-3}-1}(\xi)^8 + \cdots\\
&= \xi_1(\dlQ^{2^{n-1}m}\xi_{n-1})^2
 + \xi_2(\dlQ^{2^{n-2}m}\xi_{n-2})^4
 + \xi_3(\dlQ^{2^{n-3}m}\xi_{n-3})^8
 + \cdots \\
& \in\mathcal{I}(n-2)^{[1]} + \mathcal{I}(n-3)^{[2]} + \cdots
\quad \subseteq \mathcal{I}(n-1),
\end{align*}
and similarly $\dlQ^{2^nm+2^n-1}\zeta_n \in\mathcal{I}(n-1)$.

For $r\geq0$, $\dlQ^k(\zeta_s^{2^r})=0$ unless $2^r\mid k$,
and then by~\eqref{eq:I(n)[1]-I(n+1)},
\[
\dlQ^{2^r\ell}(\zeta_s^{2^r}) = (\dlQ^{\ell}\zeta_s)^{2^r}
\in \mathcal{I}(n-1)^{[r]} \subseteq \mathcal{I}(n+r-1).
\qedhere
\]
\end{proof}
\begin{cor}\label{cor:Qkzetas}
For $n\geq0$, the cotensor product
$\mathcal{A}_*\square_{\mathcal{A}(n)_*}\F_2\subseteq\mathcal{A}_*$
is closed under the Dyer-Lashof operations, and the Dyer-Lashof
operations commute with the Hopf algebra quotient homomorphism
$\mathcal{A}_*\to\mathcal{A}(n)_*$.
\end{cor}

\end{document}